\documentclass[12pt,a4paper,leqno]{amsart}%
\usepackage{amsfonts,amssymb,mathrsfs,bm,enumerate,color,graphicx,hyperref}
\usepackage[numbers]{natbib}
\usepackage{amsmath}
\usepackage{amsfonts}
\usepackage{amssymb}
\usepackage{graphicx}%
\providecommand{\U}[1]{\protect\rule{.1in}{.1in}}

\numberwithin{equation}{section}
\theoremstyle{plain}
\newtheorem{thm}{Theorem}[section]
\newtheorem{lem}[thm]{Lemma}
\newtheorem{cor}[thm]{Corollary}
\newtheorem{prop}[thm]{Proposition}

\theoremstyle{definition}
\newtheorem{defn}[thm]{Definition}
\newtheorem{ex}[thm]{Example}
\newtheorem{rem}[thm]{Remark}

\newcommand{\n}{\noindent}

\def\1{\scalebox{0.94}{1}\hspace{-0.35em}1}
\allowdisplaybreaks[4]
\begin{document}
\title{Type $A$ Distributions:\\Infinitely Divisible Distributions \\Related to Arcsine Density }
\date{}
\maketitle
\subjclass{ }

\begin{center}
Makoto Maejima,\footnote{Department of Mathematics, Keio University, 3-14-1,
Hiyoshi, Yokohama 223-8522, Japan} V\'{\i}ctor P\'{e}rez-Abreu
\footnote{Department of Probability and Statistics, Centro de
Investigaci\'{o}n en Matem\'{a}ticas, CIMAT, Apdo. Postal 402, Guanajuato,
Gto. 36000, Mexico} \footnote{Most of this work was done while the author
visited Keio University in Japan. He gratefully acknowledges the hospitality
and financial support during his stay} and Ken-iti Sato\footnote{Hachiman-yama
1101-5-103, Tenpaku-ku, Nagoya 468-0074, Japan}
\end{center}

\vskip 5mm

\begin{abstract}
Two transformations $\mathcal{A}_{1}$ and $\mathcal{A}_{2}$ of L\'{e}vy
measures on $\mathbb{R}^{d}$ based on the arcsine density are studied and their
relation to general Upsilon transformations is considered. The domains of
definition of $\mathcal{A}_{1}$ and $\mathcal{A}_{2}$ are determined and it is
shown that they have the same range. Infinitely divisible distributions on
$\mathbb{R}^{d}$ with L\'{e}vy measures being in the common range are called
type $A$ distributions and expressed as the law of a stochastic integral 
$\int_0^1\cos (2^{-1}\pi t)dX_t$ with
respect to L\'{e}vy process $\{X_t\}$. \ This new class includes as a proper subclass
the Jurek class of distributions. It is shown that generalized type $G$
distributions are the image of type $A$ distributions under a mapping defined
by an appropriate stochastic integral. $\mathcal{A}_{2}$ is identified as an
Upsilon transformation, while $\mathcal{A}_{1}$ is shown to be not.

\end{abstract}

\vskip 5mm
\n

\setlength{\baselineskip}{18pt} \setlength{\parindent}{1.8pc}

\vskip10mm


\section{Introduction}

Let $I(\mathbb{R }^{d})$ denote the class of all infinitely divisible
distributions on $\mathbb{R }^{d}$. For $\mu\in I(\mathbb{R}^{d})$, we use the
L\'{e}vy-Khintchine representation of its characteristic function
$\widehat{\mu}(z)$ given by
\begin{align*}
\widehat{\mu}(z)=\exp\bigg\{-\frac{1}{2}{}\langle\Sigma z,  &  z\rangle
\!+\mathrm{i} \langle\gamma,z\rangle\\
&  +\int_{\mathbb{R}^{d}}\left(  \mathrm{e}^{\mathrm{i}\langle x,z\rangle
}\!-1- \frac{\mathrm{i}\langle x.z\rangle}{1+\left\vert x\right\vert ^{2}%
}\right)  \nu(\mathrm{d}x)\bigg\},\quad z\in\mathbb{R}^{d},
\label{LevyKhintRep}%
\end{align*}
where $\Sigma$ is a symmetric nonnegative-definite $d\times d$ matrix,
$\gamma\in\mathbb{R}^{d}$ and $\nu$ is a measure on $\mathbb{R}^{d}$ (called
the L\'{e}vy measure) satisfying $\nu(\{0\})=0$ and $\int_{\mathbb{R}^{d}%
}(1\wedge\left\vert x\right\vert ^{2})\nu(\mathrm{d}x)<\infty. $ The triplet
$(\Sigma,\nu,\gamma)$ is called the L\'{e}vy-Khintchine triplet of $\mu\in
I(\mathbb{R}^{d})$. Let $\mathfrak{M}_{L}(\mathbb{R}^{d})$ denote the class of
L\'evy measures of $\mu\in I(\mathbb{R}^{d})$. The class of L\'evy measures
$\nu$ on $\mathbb{R}^{d}$ satisfying $\int_{\mathbb{R}^{d}%
}(1\wedge\left\vert x\right\vert )\nu(\mathrm{d}x)<\infty$ is denoted by
$\mathfrak{M}_{L}^{1}(\mathbb{R}^{d})$. Sometimes we write $\mathfrak{M}%
_{L}^{2}(\mathbb{R}^{d})=\mathfrak{M}_{L}(\mathbb{R}^{d})$. A measure $\nu$ on
$(0,\infty)$ is also called a L\'evy measure if it satisfies $\int%
_{(0,\infty)}(1\wedge x^{2}) \nu(\mathrm{d}x)<\infty$, and denote by
$\mathfrak{M}_{L}({(0,\infty)})$ the class of all L\'evy measures on
$(0,\infty)$.

Let
\begin{equation}
a(x;s)=%
\begin{cases}
\pi^{-1}(s-x^{2})^{-1/2}, & \quad|x|<{s}^{1/2},\\
0, & \quad\text{otherwise},
\end{cases}
\label{arcs0}%
\end{equation}
which is the density of the symmetric arcsine law with parameter $s>0$. For a
L\'{e}vy measure $\rho\in\mathfrak{M}_{L}((0,\infty))$, define
\begin{equation}
\ell(x)=\int_{\mathbb{R}_{+}}a(x;s)\rho(\mathrm{d}s),\quad x\in\mathbb{R}.
\label{Levden}%
\end{equation}
In \cite{BNPA08}, a distribution such that its L\'{e}vy measure is either zero
or has a density $\ell$ of the form (\ref{Levden}) is called a \textit{type
$A$ distribution} on $\mathbb{R}$. It is a one-dimensional symmetric
distribution. Let $Z$ be a standard normal random variable and $V$ a positive
infinitely divisible random variable independent of $Z$. The distribution of
the one-dimensional random variable $V^{1/2}Z$ is infinitely divisible and is
called of type $G$. It is also shown in \cite{BNPA08} that an infinitely
divisible distribution $\widetilde{\mu}$ on $\mathbb{R}$ is of type $G$ if and
only if there exists a type $A$ distribution $\mu$ on $\mathbb{R}$ with the
following stochastic integral mapping representation
\begin{equation}
\widetilde{\mu}=\mathcal{L}\left(  \int_{0}^{1/2}\left(  \log\frac{1}%
{t}\right)  ^{1/2}dX_{t}^{(\mu)}\right)  . \label{Grep}%
\end{equation}
Here and in what follows, $\mathcal{L}$ means \textquotedblleft the law of"
and $\{X_{t}^{(\mu)}\}$ means a L\'{e}vy process on $\mathbb{R}^{d}$ whose
distribution at time 1 is $\mu\in I(\mathbb{R}^{d})$. (In \eqref{Grep}, $d=1$.)

In this paper, we study more about type $A$ (not necessarily symmetric)
distributions on $\mathbb{R}^{d}$. The organization of this paper is the following.

Section 2 considers two arcsine transformations $\mathcal{A}_{1}$ and
$\mathcal{A}_{2}$ of L\'{e}vy measures on $\mathbb{R}^{d}$ based on
(\ref{arcs0}) and a reparametrization of this density of the arcsine law. It
is shown that the domains of the transformations $\mathcal{A}_{1}$ and
$\mathcal{A}_{2}$ are $\mathfrak{M}_{L}^{1}(\mathbb{R}^{d})$ and
$\mathfrak{M}_{L}^{2}(\mathbb{R}^{d}),$ respectively, but they are identical
modulo some $(p)$-transformation. We see that both transformations are
one-to-one and that they have the same range $\mathfrak{R}(\mathcal{A}_{k}).$
It is shown that this range contains as a proper subclass the Jurek class
$U(\mathbb{R}^{d})$ of distributions on $\mathbb{R}^{d}$ studied in
\cite{Ju85}, \cite{MS08}. $U(\mathbb{R}^{d})$ includes several known classes
of multivariate distributions characterized by the radial part of their
L\'{e}vy measures, such as the Goldie-Steutel-Bondesson class $B(\mathbb{R}%
^{d})$, the class of selfdecomposable distributions $L(\mathbb{R}^{d})$ and
the Thorin class $T(\mathbb{R}^{d})$, see \cite{BNMS06}. Recently, other
bigger classes than the Jurek class have been discussed in the study of
extension of selfdecomposability, see \cite{MMS10} and \cite{Sato10}.

Section 3 deals with the class $A(\mathbb{R}^{d})$ of type $A$ distributions
on $\mathbb{R}^{d}$ defined as those infinitely divisible distributions on
$\mathbb{R}^{d}$ which L\'{e}vy measure ${\nu}$ belongs to $\mathfrak{R}%
(\mathcal{A}_{k})$.\ Some probabilistic interpretations are considered and the
relation to the class $G(\mathbb{R}^{d})$ of generalized type $G$
distributions on $\mathbb{R}^{d}$ introduced in \cite{MS08} is studied. It is
shown that $A(\mathbb{R}^{d})=\Phi_{\cos}(I(\mathbb{R}^{d})),$ where
$\Phi_{\cos}$ is the stochastic integral mapping
\[
\Phi_{\cos}(\mu)=\mathcal{L}\left(  \int_{0}^{1}\cos(2^{-1}\pi t)\mathrm{d}%
X_{t}^{(\mu)}\right)  ,\quad\mu\in I(\mathbb{R}^{d}).
\]
It is also shown that the class of \ L\'{e}vy measure in $G(\mathbb{R}^{d})$
is the image of the class of \ L\'{e}vy measures in $B(\mathbb{R}^{d}%
)\cap\mathfrak{M}_{L}^{1}(\mathbb{R}^{d})$ under $\mathcal{A}_{1}$. \ In order
to prove this, and as a result of independent interest, a new arcsine
representation of completely monotone functions is first obtained. In
addition, the class $G(\mathbb{R}^{d})$ is described as the image of
$A(\mathbb{R}^{d})$ under the stochastic integral mapping (\ref{Grep}),
$d\geq1,$ including the multivariate and non-symmetric cases. For doing this,
we first have to prove that $\mathcal{A}_{2}$ is an Upsilon transformation in
the sense of \cite{BNRT08}. However, we see that, remarkably, $\mathcal{A}%
_{1}$ is not an Upsilon transformation and it is not commuting with a specific
Upsilon transformation, which is different form other cases considered so far.
Finally, Section 4 contains examples of $\mathcal{A}_{1}$ and $\mathcal{A}%
_{2}$ transformations of L\'{e}vy measures where the modified Bessel function
$K_{0}$ plays an important role.

\vskip 10mm


\section{Two arcsine transformations $\mathcal{A}_{1}$ and $\mathcal{A}_{2}$
on $\mathbb{R }^{d}$}

\vskip 3mm

\subsection{Definitions and domains}


Besides the arcsine density (\ref{arcs0}), we consider two one-sided arcsine
densities with different parameters $s>0$ and $s^{2}$ as follows:
\[
a_{1}(r;s)=
\begin{cases}
2\pi^{-1}(s-r^{2})^{-1/2}, & \quad0<r<s^{1/2},\\
0, & \quad\text{otherwise},
\end{cases}
\label{arcs1}%
\]
and
\[
a_{2}(r;s)=
\begin{cases}
2\pi^{-1}(s^{2}-r^{2})^{-1/2}, & \quad0<r<s,\\
0, & \quad\text{otherwise}.
\end{cases}
\label{arcs2}%
\]
Then we consider two arcsine transformations $\mathcal{A}_{1}$ and
$\mathcal{A}_{2}$ of measures on $\mathbb{R }^{d}$ based on these two
one-sided arcsine densities.

\begin{defn}
\label{d1} Let $\nu$ be a measure on $\mathbb{R}^{d}$ satisfying
$\nu(\{0\})=0$. Then, for $k=1,2$, define the \emph{arcsine transformation}
$\mathcal{A}_{k}$ of $\nu$ by
\begin{equation}
\mathcal{A}_{k}(\nu)(B)=\int_{\mathbb{R}^{d}\setminus\{0\}}\nu(\mathrm{d}x)
\int_{0}^{\infty}a_{k}(r;|x|)1_{B}\left(  r\frac{x}{|x|}\right)
\mathrm{d}r,\quad B\in\mathcal{B}(\mathbb{R}^{d}). \label{DefArcTrn1}%
\end{equation}

\end{defn}

The domain $\mathfrak{D}(\mathcal{A}_{k})$ is the class of measures $\nu$ on
$\mathbb{R}^{d}$ such that $\nu(\{0\})=0$ and the right-hand side of
\eqref{DefArcTrn1} is a L\'evy measure in $\mathfrak{M}_{L}(\mathbb{R}^{d})$.
The range is
\[
\mathfrak{R}(\mathcal{A}_{k})=\{\mathcal{A}_{k}(\nu)\colon\nu\in
\mathfrak{D}(\mathcal{A}_{k})\}.
\]

\vskip 3mm We have the following result about domains of $\mathcal{A}_{1}$ and
$\mathcal{A}_{2}$.

\begin{thm}
\label{DomArcTrn} The domains of $\mathcal{A}_{k}$ are as follows:
\[
\mathfrak{D}(\mathcal{A}_{k})=\mathfrak{M}_{L}^{k}(\mathbb{R}^{d}), \quad
k=1,2.
\]

\end{thm}

\noindent\textit{Proof.} We write $c=2\pi^{-1}$. First, let us show that
$\mathfrak{D}(\mathcal{A}_{k}) \subset\mathfrak{M}_{L}^{k}(\mathbb{R}^{d})$.
Suppose that $\nu\in\mathfrak{D}(\mathcal{A}_{k})$. Write $\widetilde{\nu}%
_{k}=\mathcal{A}_{k}(\nu)$. Then
\begin{equation}
\label{DefArcTrn3}\widetilde{\nu}_{k}(B)=\int_{\mathbb{R}^{d}\setminus\{0\}}
\nu(\mathrm{d}x) \int_{0}^{|x|^{k/2}}c(|x|^{k} -r^{2})^{-1/2} 1_{B}\left(
r\frac{x}{|x|}\right)  \mathrm{d} r.
\end{equation}
Hence, for all nonnegative measurable functions $f$ on $\mathbb{R}%
^{d}\setminus\{0\}$,
\[
\int_{\mathbb{R}^{d}\setminus\{0\}} f(x)\widetilde{\nu}_{k}(\mathrm{d}x)
=\int_{\mathbb{R}^{d}\setminus\{0\}} \nu(\mathrm{d}x) \int_{0}^{|x|^{k/2}}
c(|x|^{k} -r^{2})^{-1/2} f\left(  r\frac{x}{|x|}\right)  \mathrm{d} r.
\]
In particular,
\begin{equation}
\label{DefArcTrn5}\int_{\mathbb{R}^{d}}(1\land|x|^{2})\widetilde{\nu}%
_{k}(\mathrm{d}x) =c\int_{\mathbb{R}^{d}} \nu(\mathrm{d}x) \int_{0}%
^{|x|^{k/2}} (|x|^{k} -r^{2})^{-1/2}(1\land r^{2})\mathrm{d} r.
\end{equation}
Since $\int(1\land|x|^{2})\widetilde{\nu}_{k}(\mathrm{d}x)<\infty$, we see
that
\begin{align*}
\infty &  > c\int_{\mathbb{R}^{d}}\nu(\mathrm{d}x) \int_{0}^{1} (1-u^{2}%
)^{-1/2}(1\land(|x|^{k} u^{2})) \mathrm{d} u\\
&  \geq c\int_{\mathbb{R}^{d}} (1\land|x|^{k})\nu(\mathrm{d}x) \int_{0}%
^{1}(1-u^{2})^{-1/2} u^{2}\mathrm{d} u.
\end{align*}
Hence $\nu\in\mathfrak{M}_{L}^{k}(\mathbb{R}^{d})$.

Next let us show that $\mathfrak{M}_{L}^{k}(\mathbb{R}^{d})\subset
\mathfrak{D}(\mathcal{A}_{k})$. Suppose that $\nu\in\mathfrak{M}_{L}%
^{k}(\mathbb{R}^{d})$. Let $\widetilde{\nu}_{k}(B)$ denote the right-hand side
of \eqref{DefArcTrn1}. Then $\widetilde{\nu}_{k}$ is a measure on
$\mathbb{R}^{d}$ with $\widetilde{\nu}_{k}(\{0\})=0$ and \eqref{DefArcTrn3}
and \eqref{DefArcTrn5} hold. Hence
\begin{align*}
&  \int_{\mathbb{R}^{d}}(1\land|x|^{2})\widetilde{\nu}_{k}(\mathrm{d}x)
=c\int_{\mathbb{R}^{d}} \nu(\mathrm{d}x) \int_{0}^{1} (1-u^{2})^{-1/2}
(1\land(|x|^{k} u^{2}))\mathrm{d} u\\
&  \qquad\leq c\int_{\mathbb{R}^{d}} (1\land|x|^{k}) \nu(\mathrm{d}x) \int%
_{0}^{1} (1-u^{2})^{-1/2}\mathrm{d} u<\infty.
\end{align*}
This shows that $\nu\in\mathfrak{D}(\mathcal{A}_{k})$. \qed

\medskip In order to study the relation between $\mathcal{A}_{1}$ and
$\mathcal{A}_{2}$, we use the following transformation of measures.

\begin{defn}
\label{p-transf} Let $p>0$. For any measure $\rho$ on $(0,\infty)$, define a
measure ${\rho}^{(p)}$ on $(0,\infty)$ by
\[
\label{p-transf-1}{\rho}^{(p)}(E)=\int_{(0,\infty)}1_{E}(s^{p})\rho
(\mathrm{d}s),\qquad E\in\mathcal{B}((0,\infty)).
\]
More generally, for any measure $\nu$ on $\mathbb{R}^{d}$ with $\nu(\{0\})=0$,
define a measure $\nu^{(p)}$ on $\mathbb{R}^{d}$ by
\[
\label{p-transf-2}{\nu}^{(p)}(B)=\int_{\mathbb{R}^{d}\setminus\{0\}}%
1_{B}\left(  |x|^{p}\frac{x}{|x|}\right)  \nu(\mathrm{d}x),\qquad
B\in\mathcal{B}(\mathbb{R}^{d}).
\]
We call the mapping from $\nu$ to $\nu^{(p)}$ $\left(  \emph{p}\right)
$\emph{-transformation}.
\end{defn}

The following result is the polar decomposition of a L\'{e}vy measure in
$\mathfrak{M}_{L}(\mathbb{R }^{d})$, (see \cite{BNMS06}, \cite{Ro90}). Here we
include the case $\nu=0$. It is a basic tool to study multivariate infinitely
divisible distributions.

\begin{prop}
\label{PolDec} Let $\nu\in\mathfrak{M}_{L}(\mathbb{R}^{d})$. Then there exists
a measure $\lambda$ on the unit sphere $\mathbb{S}=\{\xi\in\mathbb{R}^{d}
\colon\left\vert \xi\right\vert =1\}$ with $0\leq\lambda(\mathbb{S})\leq
\infty$ and a family $\left\{  \nu_{\xi}\colon\xi\in\mathbb{S}\right\}  $ of
measures on $(0,\infty)$ such that $\nu_{\xi}(E)$ is measurable in $\xi$ for
each $E\in\mathcal{B}((0,\infty)),$ $0<\nu_{\xi}((0,\infty))\leq\infty$ for
each $\xi\in\mathbb{S}$, and
\[
\nu(B)=\int_{\mathbb{S}}\lambda(\mathrm{d}\xi)\int_{0}^{\infty} 1_{B}(r\xi
)\nu_{\xi}(\mathrm{d}r),\quad B\in\mathcal{B}(\mathbb{R}^{d}).\label{PolDecLM}%
\]
$\nu_{\xi}$ is called the radial component of $\nu$. Here $\lambda$ and
$\{\nu_{\xi}\}$ are uniquely determined by $\nu$ in the following sense: if
$(\lambda,\nu_{\xi})$ and $(\lambda^{\prime},\nu^{\prime}_{\xi})$ both have
the properties above, then there is a measurable function $c(\xi)$ on
$\mathbb{S}$ such that
\begin{gather*}
0<c(\xi)<\infty,\\
\lambda^{\prime}(\mathrm{d}\xi)=c(\xi)\lambda(\mathrm{d}\xi),\\
c(\xi)\nu^{\prime}_{\xi}(\mathrm{d} r)=\nu_{\xi}(\mathrm{d} r)\quad\text{for
$\lambda$-a.\,e.\ $\xi$}.
\end{gather*}

\end{prop}

If $\nu\in\mathfrak{M}_{L}(\mathbb{R}^{d})$ has a polar decomposition
$(\lambda,\nu_{\xi})$, then
\[
\nu^{(p)}(B)=\int_{\mathbb{S}}\lambda(\mathrm{d}\xi)\int_{0}^{\infty}
1_{B}(r\xi)\,{\nu_{\xi}}^{(p)}(\mathrm{d}r),\qquad B\in\mathcal{B}
(\mathbb{R}^{d}).
\]
If $\widetilde{\nu}=\nu^{(p)}$, then $\nu={\widetilde{\nu}}^{(1/p)}$. For any
nonnegative measurable function $f(x)$ on $\mathbb{R}^{d}$,
\begin{equation}
\int_{\mathbb{R}^{d}}f(x)\nu^{(p)}(\mathrm{d}x)=\int_{\mathbb{R}^{d}%
\setminus\{0\}}f(|x|^{p-1}x)\nu(\mathrm{d}x). \label{p-transf-4}%
\end{equation}

\vskip3mm The two arcsine transformations are identical modulo some $\left(
p\right)  $-transformations.

\begin{prop}
\label{p-trRep} $\nu\in\mathfrak{M}_{L}^{2}(\mathbb{R}^{d})$ if and only if
$\nu^{(2)}\in\mathfrak{M}_{L}^{1}(\mathbb{R}^{d})$, and in this case
\[
\label{p-trRep-1}\mathcal{A}_{2}(\nu)=\mathcal{A}_{1}(\nu^{(2)}).
\]
Also $\nu\in\mathfrak{M}_{L}^{1}(\mathbb{R}^{d})$ if and only if $\nu
^{(1/2)}\in\mathfrak{M}_{L}^{2}(\mathbb{R}^{d})$, and in this case
\[
\label{p-trRep-2}\mathcal{A}_{1}(\nu)=\mathcal{A}_{2}(\nu^{(1/2)}).
\]

\end{prop}

\noindent\textit{Proof.} Equivalence of $\nu\in\mathfrak{M}_{L}^{2}%
(\mathbb{R}^{d})$ and $\nu^{(2)} \in\mathfrak{M}_{L}^{1}(\mathbb{R}^{d})$
follows from
\[
\int_{\mathbb{R }^{d}}(1\land|x|)\,\nu^{(2)}(\mathrm{d}x)= \int_{\mathbb{R}%
^{d}}(1\land|x|^{2})\,\nu(\mathrm{d}x)
\]
obtained from \eqref{p-transf-4}. We have
\begin{align*}
\mathcal{A}_{1}(\nu^{(2)})(B)  &  =\int_{\mathbb{S}}\lambda(\mathrm{d}\xi)
\int_{0}^{\infty} 1_{B}(r\xi)\mathrm{d}r \int_{r^{2}}^{\infty} 2\pi^{-1}
(s-r^{2})^{-1/2}\,{\nu_{\xi}}^{(2)} (\mathrm{d}s)\\
&  = \int_{\mathbb{S}}\lambda(\mathrm{d}\xi) \int_{0}^{\infty} 1_{B}
(r\xi)\mathrm{d}r \int_{r}^{\infty} 2\pi^{-1} (s^{2}-r^{2})^{-1/2}\,\nu_{\xi}
(\mathrm{d}s)\\
&  =\mathcal{A}_{2}(\nu)(B),
\end{align*}
proving the first half. The proof of the second half is similar. \qed

\vskip 5mm

\subsection{One-to-one property}


We next show that the arcsine transformations $\mathcal{A}_{1}$ and
$\mathcal{A}_{2}$ are one-to-one. In contrast to usual proofs for the
one-to-one property by the use of Laplace transform, our proof here has a
different flavor.

Let us first prove some lemmas. A measure $\sigma$ on $(0,\infty)$ is said to
be locally finite on $(0,\infty)$ if $\sigma((b,c))<\infty$ whenever
$0<b<c<\infty$. For a measure $\rho$ on $(0,\infty)$, define
\[
\mathfrak{A}(\rho)(\mathrm{d}u)=\left(  \int_{(u,\infty)} \pi^{-1/2}%
(s-u)^{-1/2}\rho(\mathrm{d}s)\right)  \mathrm{d}u ,\label{d1.1}%
\]
if the integral in the right-hand side is the density of a locally finite
measure on $(0,\infty)$. This is fractional integral of order $1/2$.

\begin{lem}
\label{lem1} If
\begin{equation}
\label{lem1.1}\int_{(b,\infty)} s^{-1/2} \rho(\mathrm{d}s)<\infty
\quad\text{for all\/ $b>0 $},
\end{equation}
then $\mathfrak{A }(\rho)$ is definable.
\end{lem}

\begin{proof}
Let $0<b<c<\infty$. We have
\begin{align*}
&  \int_{b}^{c} \mathrm{d}u\int_{u}^{\infty}(s-u)^{-1/2}\rho(\mathrm{d}s)
=\int_{b}^{\infty}\rho(\mathrm{d}s)\int_{b}^{c\land s} (s-u)^{-1/2}%
\mathrm{d}u\\
&  \qquad=\int_{b}^{c} \rho(ds)\int_{b}^{s} (s-u)^{-1/2}\mathrm{d}u +\int%
_{c}^{\infty}\rho(\mathrm{d}s)\int_{b}^{c} (s-u)^{-1/2}\mathrm{d}u\\
&  \qquad=2\int_{b}^{c} (s-b)^{1/2}\rho(\mathrm{d}s)+ 2 \int_{c}^{\infty
}((s-b)^{1/2}-(s-c)^{1/2})\rho(\mathrm{d}s),
\end{align*}
which is finite, since $(s-b)^{1/2}-(s-c)^{1/2}\sim(c-b)s^{-1/2}$ as
$s\to\infty$.
\end{proof}

\begin{lem}
\label{lem1a} Suppose that $\mathfrak{A}(\rho)$ is definable. Then, for
$\alpha>-1$ and $b>0$,
\begin{equation}
\label{lem1a-1}\int_{(b,\infty)} u^{\alpha}\mathfrak{A}(\rho)(\mathrm{d}u)\leq
C_{1}\int_{(b,\infty)} s^{\alpha+1/2}\rho(\mathrm{d}s)
\end{equation}
and
\begin{equation}
\label{lem1a-2}\int_{(0,b]} u^{\alpha}\mathfrak{A}(\rho)(\mathrm{d}u)\leq
C_{2}\left(  \int_{(0,b]} s^{\alpha+1/2}\rho(\mathrm{d}s)+\int_{(b,\infty)}
s^{-1/2}\rho(\mathrm{d}s)\right)  ,
\end{equation}
where $C_{1}$ and $C_{2}$ are constants independent of $\rho$.
\end{lem}

\begin{proof}
Let $c=\pi^{-1/2}$. We have
\begin{align*}
&  \int_{(b,\infty)}u^{\alpha}\mathfrak{A}(\rho)(\mathrm{d}u) =c\int%
_{b}^{\infty}u^{\alpha}\mathrm{d}u\int_{(u,\infty)}(s-u)^{-1/2}\rho
(\mathrm{d}s)\\
&  \qquad=c\int_{(b,\infty)}\rho(\mathrm{d}s)\int_{b}^{s}u^{\alpha
}(s-u)^{-1/2}\mathrm{d}u
\end{align*}
and
\begin{align*}
&  \int_{b}^{s}u^{\alpha}(s-u)^{-1/2}\mathrm{d}u=s^{-1/2}\int_{b}^{s}
u^{\alpha}(1-s^{-1}u)^{-1/2}\mathrm{d}u\\
&  \qquad=s^{\alpha+1/2}\int_{b/s}^{1}v^{\alpha}(1-v)^{-1/2}\mathrm{d}v \sim
s^{\alpha+1/2}B(\alpha+1,1/2),\quad s\rightarrow\infty.
\end{align*}
Hence \eqref{lem1a-1} holds. We have
\begin{align*}
&  \int_{(0,b]}u^{\alpha}\mathfrak{A}(\rho)(\mathrm{d}u) =c\int_{0}%
^{b}u^{\alpha}\mathrm{d}u\int_{(u,\infty)}(s-u)^{-1/2}\rho(\mathrm{d}s)\\
&  \quad=c\int_{(0,\infty)}\rho(\mathrm{d}s) \int_{0}^{s\wedge b}u^{\alpha
}(s-u)^{-1/2}\mathrm{d}u\\
&  \quad=c\int_{(0,b]}\rho(\mathrm{d}s)\int_{0}^{s}u^{\alpha}(s-u)^{-1/2}
\mathrm{d}u+c\int_{(b,\infty)}\rho(\mathrm{d}s)\int_{0}^{b}u^{\alpha
}(s-u)^{-1/2}\mathrm{d}u.
\end{align*}
Notice that
\[
\int_{0}^{s}u^{\alpha}(s-u)^{-1/2}\mathrm{d}u=s^{\alpha+1/2}B(\alpha+1,1/2)
\]
and
\begin{align*}
&  \int_{0}^{b}u^{\alpha}(s-u)^{-1/2}\mathrm{d}u=s^{-1/2}\int_{0}^{b}
u^{\alpha}(1-u/s)^{-1/2}\mathrm{d}u\\
&  \qquad\leq s^{-1/2}\int_{0}^{b}u^{\alpha}(1-u/b)^{-1/2}\mathrm{d}
u=s^{-1/2}b^{\alpha+1}B(\alpha+1,1/2),\quad s>b.
\end{align*}
Thus \eqref{lem1a-2} holds.
\end{proof}

\begin{lem}
\label{lem2} Suppose that
\begin{equation}
\rho((b,\infty))<\infty\quad\text{for all\/ $b>0$}. \label{lem2.1}%
\end{equation}
Then $\mathfrak{A}(\rho)$ and $\mathfrak{A}(\mathfrak{A}(\rho))$ are definable
and
\begin{equation}
\mathfrak{A}(\mathfrak{A}(\rho))(\mathrm{d}u)=\rho((u,\infty))\,\mathrm{d}u,
\label{lem2.2}%
\end{equation}
which implies that $\rho$ is determined by $\mathfrak{A}(\rho)$ under the
condition (\ref{lem2.1}).
\end{lem}

\noindent\textit{Proof.} Since \eqref{lem2.1} is stronger than \eqref{lem1.1},
$\mathfrak{A}(\rho)$ is definable. Using \eqref{lem1a-1} of Lemma \ref{lem1a},
we see from Lemma \ref{lem1} that $\mathfrak{A}(\mathfrak{A}(\rho))$ is
definable. Next, notice that
\begin{align*}
&  \int_{u}^{\infty}\pi^{-1/2}(s-u)^{-1/2}\,\mathfrak{A}(\rho)(\mathrm{d}s)\\
&  \qquad=\pi^{-1}\int_{u}^{\infty}(s-u)^{-1/2}\mathrm{d}s\int_{(s,\infty
)}(v-s)^{-1/2}\rho(\mathrm{d}v)\\
&  \qquad=\pi^{-1}\int_{(u,\infty)}\rho(\mathrm{d}v)\int_{u}^{v}
(s-u)^{-1/2}(v-s)^{-1/2}\mathrm{d}s=\rho((u,\infty)),
\end{align*}
because
\[
\int_{u}^{v}(s-u)^{-1/2}(v-s)^{-1/2}\mathrm{d}s=\int_{0}^{1}s^{-1/2}
(1-s)^{-1/2}\mathrm{d}s=B(1/2,1/2)=\pi.
\]
Hence \eqref{lem2.2} is true. \qed

\vskip 3mm For the proof of the next theorem, we introduce new functions for
simplicity. For any measure $\rho$ on $(0,\infty)$ and for $k=1,2$, let
\[
\mathsf{a}_{k}(\rho)(r)=\int_{(0,\infty)}a_{k}(r;s)\,\rho(\mathrm{d}%
s),\label{DefArcTrnMes}%
\]
admitting the infinite value.

\begin{thm}
\label{t2} For $k=1,2$, $\mathcal{A}_{k}$ is one-to-one.
\end{thm}

\noindent\textit{Proof.} Case $k=1$. Suppose that $\nu,\nu^{\prime}%
\in\mathfrak{M}_{L}^{1}(\mathbb{R}^{d})$ and $\mathcal{A}_{1}(\nu
)=\mathcal{A}_{1} (\nu^{\prime})$. Let $(\lambda, \nu_{\xi})$ and
$(\lambda^{\prime}, \nu^{\prime}_{\xi})$ be polar decompositions of $\nu$ and
$\nu^{\prime}$, respectively. Then {\allowdisplaybreaks
\begin{align*}
\mathcal{A}_{1}(\nu)(B)  &  =\int_{\mathbb{S}}\lambda(\mathrm{d}\xi) \int%
_{0}^{\infty} 1_{B}(r\xi)\,\mathsf{a}_{1}(\nu_{\xi})(r)\mathrm{d}r,\\
\mathcal{A}_{1}(\nu^{\prime})(B)  &  =\int_{\mathbb{S}} \lambda^{\prime
}(\mathrm{d}\xi)\int_{0}^{\infty} 1_{B}(r\xi)\,\mathsf{a}_{1} (\nu^{\prime
}_{\xi})(r)\mathrm{d}r.
\end{align*}
Hence it follows from Proposition \ref{PolDec} that there is a measurable
function} $c(\xi)$ satisfying $0<c(\xi)<\infty$ such that $\lambda^{\prime
}(\mathrm{d}\xi)=c(\xi)\lambda(\mathrm{d}\xi)$ and $\mathsf{a}_{1}(\nu
^{\prime}_{\xi})(r)\mathrm{d}r=c(\xi)^{-1}\mathsf{a}_{1}(\nu_{\xi
})(r)\mathrm{d}r$ for $\lambda$-a.\,e.\ $\xi$. Thus
\[
\left(  \int_{r^{2}}^{\infty} (s-r^{2})^{-1/2}\nu^{\prime}_{\xi}%
(\mathrm{d}s)\right)  \mathrm{d}r= \left(  c(\xi)^{-1}\int_{r^{2}}^{\infty}
(s-r^{2})^{-1/2} \nu_{\xi}(\mathrm{d}s)\right)  \mathrm{d}r.
\]
Using a new variable $u=r^{2}$, we see that
\[
\left(  \int_{u}^{\infty} (s-u)^{-1/2}\nu^{\prime}_{\xi}(\mathrm{d}s)\right)
\mathrm{d}u= \left(  c(\xi)^{-1}\int_{u}^{\infty} (s-u)^{-1/2}\nu_{\xi}
(\mathrm{d}s)\right)  \mathrm{d}u.
\]
Since $\nu_{\xi}$ and $\nu^{\prime}_{\xi}$ satisfy \eqref{lem2.1}, we obtain
$\nu_{\xi}=c(\xi)^{-1} \nu^{\prime}_{\xi}$ for $\lambda$-a.\,e.\ $\xi$ from
Lemma \ref{lem2}. It follows that $\nu=\nu^{\prime}$.

Case $k=2$. Use Proposition \ref{p-trRep}. Then $\mathcal{A}_{2}(\nu)$ equals
$\mathcal{A}_{1}(\nu^{(2)})$, which determines $\nu^{(2)}$ by Case $k=1$, and
$\nu^{(2)}$ determines $\nu=(\nu^{(2)})^{(1/2)}$. \qed

\vskip 5mm

\subsection{Ranges}


We will show some facts concerning the ranges of $\mathcal{A}_{1}$ and
$\mathcal{A}_{2}$.

\begin{prop}
\label{Range1} The ranges of $\mathcal{A}_{1}$ and $\mathcal{A}_{2}$ are
identical:
\[
\label{Range1-1}\mathfrak{R}(\mathcal{A}_{1})=\mathfrak{R}(\mathcal{A}_{2}).
\]

\end{prop}

\noindent\textit{Proof.} This is a direct consequence of Proposition
\ref{p-trRep}. \qed

\vskip 3mm Let us show some necessary conditions for $\widetilde{\nu}$ to
belong to the range.

\begin{prop}
\label{Range2} If\/ $\widetilde{\nu}$ is in the common range of $\mathcal{A}%
_{1}$ and $\mathcal{A}_{2}$, then $\widetilde{\nu}$ is in $\mathfrak{M}%
_{L}(\mathbb{R}^{d})$ with a polar decomposition $(\lambda, \ell_{\xi}(r)dr)$
having the following properties: $\ell_{\xi}(r)$ is measurable in $(\xi,r)$
and lower semi-continuous in $r\in(0,\infty)$, and there is $b_{\xi}%
\in(0,\infty]$ such that $\ell_{\xi}(r)>0$ for $r<b_{\xi}$ and, if\/ $b_{\xi
}<\infty$, then $\ell_{\xi}(r)=0$ for $r\geq b_{\xi}$.
\end{prop}

\noindent\textit{Proof.} Let $\widetilde{\nu}=\mathcal{A}_{k}(\nu)$ with
$\nu\in\mathfrak{M}_{L}^{k}(\mathbb{R}^{d})$ and $(\lambda,\nu_{\xi})$ a polar
decomposition of $\nu$. Then $\widetilde{\nu}\in\mathfrak{M}_{L}%
(\mathbb{R}^{d})$ with polar decomposition $(\lambda,\mathsf{a}_{k}(\nu_{\xi
})(r)\mathrm{d}r)$ from the definition. Recall that
\[
\mathsf{a}_{k}(\nu_{\xi})(r)=2\pi^{-1}\int_{(r^{2/k},\infty)}(s^{k}%
-r^{2})^{-1/2}\nu_{\xi}(\mathrm{d}s).
\]
Then our assertion is proved in the same way as Proposition 2.13 of
\cite{Sato10}. \qed

\vskip 5mm

\subsection{How big is $\mathfrak{R}(\mathcal{A}_{k})$?}


Several well-known and well studied classes of multivariate infinitely
divisible distributions are the following. The Jurek class, the class of
selfdecomposable distributions, the Goldie-Steutel-Bondesson class, the Thorin
class and the class of generalized type $G$ distributions. They are
characterized only by the radial component of their L\'{e}vy measures with no
influence of $\Sigma$ and $\gamma$ in the L\'{e}vy-Khintchine triplet. Among
them, the Jurek class is the biggest. Recently, bigger than the Jurek class
have been discussed in the study of extension of selfdecomposability, (see,
e.g. \cite{MMS10} and \cite{Sato10}). Then a natural question is how big
$\mathfrak{R }(\mathcal{A}_{k})$ is. Let $\mathfrak{M}_{L}^{U}(\mathbb{R}%
^{d})$ be the class of L\'{e}vy measures of distributions in the Jurek class.
The radial component $\nu_{\xi}$ of $\nu\in\mathfrak{M}_{L}^{U}(\mathbb{R}%
^{d})$ satisfies that $\nu_{\xi}(\mathrm{d}r)=\ell_{\xi}(r)\mathrm{d}r,r>0,$
where $\ell_{\xi}(r)$ is measurable in $(\xi,r)$ and decreasing and
right-continuous in $r>0$. We will show below that $\mathfrak{R }%
(\mathcal{A}_{k})$ is at least strictly bigger than $\mathfrak{M}_{L}%
^{U}(\mathbb{R}^{d})$.

\begin{thm}
\label{Range3} We have
\[
\mathfrak{M}_{L}^{U}(\mathbb{R}^{d})\subsetneqq\mathfrak{R}(\mathcal{A}_{1})
=\mathfrak{R}(\mathcal{A}_{2}).\label{Range3-1}%
\]

\end{thm}

\noindent\textit{Proof.} Let $\widetilde{\nu}\in\mathfrak{M}_{L}%
^{U}(\mathbb{R}^{d})$. Equivalently, let $\widetilde{\nu}\in\mathfrak{M}%
_{L}(\mathbb{R}^{d})$ with a polar decomposition $(\lambda,\ell_{\xi
}(r)\mathrm{d}r)$ such that $\ell_{\xi}(r)$ is measurable in $(\xi,r)$ and
decreasing and right-continuous in $r>0$. Further, we may and do assume that
$\lambda$ is a probability measure and
\[
\int_{0}^{\infty}(1\wedge r^{2})\ell_{\xi}(r)\mathrm{d}r=c:=\int%
_{\mathbb{R}^{d}}(1\wedge|x|^{2})\widetilde{\nu}(\mathrm{d}x).
\]
Let $\rho_{\xi}$ be a measure on $(0,\infty)$ such that $\rho_{\xi}%
((r^{2},\infty))=\ell_{\xi}(r)$ for $r>0$ and let $\eta_{\xi}=\mathfrak{A}%
(\rho_{\xi})$. Lemma \ref{lem2} says that $\eta_{\xi}$ is definable and
\[
\rho_{\xi}((u,\infty))=\int_{(u,\infty)}\pi^{-1/2}(s-u)^{-1/2}\eta_{\xi
}(\mathrm{d}s)\quad\text{for Lebesgue a.\thinspace e.\ $u>0$}.
\]
Note that $\eta_{\xi}(E)$ is measurable in $\xi$ for each $E\in\mathcal{B}%
((0,\infty))$. We have, for $B\in\mathcal{B}(\mathbb{R}^{d})$,
\begin{align*}
\widetilde{\nu}(B)  &  =\int_{\mathbb{S}}\lambda(\mathrm{d}\xi)\int%
_{0}^{\infty}1_{B}(r\xi)\rho_{\xi}((r^{2},\infty))\mathrm{d}r\\
&  =\int_{\mathbb{S}}\lambda(\mathrm{d}\xi)\int_{0}^{\infty}1_{B}%
(r\xi)\mathrm{d}r\int_{(r^{2},\infty)}\pi^{-1/2}(s-r^{2})^{-1/2}\eta_{\xi
}(\mathrm{d}s)\\
&  =\int_{\mathbb{S}}\lambda(\mathrm{d}\xi)\int_{0}^{\infty}1_{B}(r\xi
)(\pi^{1/2}/2)\mathsf{a}_{1}(\eta_{\xi})(r)\mathrm{d}r.
\end{align*}
We claim that
\begin{equation}
\int_{\mathbb{S}}\lambda(\mathrm{d}\xi)\int_{0}^{\infty}(1\wedge u)\eta_{\xi
}(\mathrm{d}u)<\infty. \label{Range3-2}%
\end{equation}
This will ensure that $(\lambda,(\pi^{1/2}/2)\eta_{\xi}(\mathrm{d}r))$ is a
polar decomposition of some $\nu\in\mathfrak{M}_{L}^{1}(\mathbb{R}^{d})$ and
that $\widetilde{\nu}=\mathcal{A}_{1}(\nu)$. First, notice that
\begin{align*}
c  &  =\int_{0}^{\infty}(1\wedge r^{2})\rho_{\xi}((r^{2},\infty))\mathrm{d}%
r=\frac{1}{2}\int_{0}^{\infty}(1\wedge u)\rho_{\xi}((u,\infty))u^{-1/2}%
\mathrm{d}u\\
&  =\frac{1}{2}\int_{0}^{1}u^{1/2}\rho_{\xi}((u,\infty))\mathrm{d}u+\frac
{1}{2}\int_{1}^{\infty}u^{-1/2}\rho_{\xi}((u,\infty))\mathrm{d}u\\
&  \geq\frac{1}{3}\rho_{\xi}((1,\infty))+\frac{1}{2}\int_{1}^{\infty}%
u^{-1/2}\rho_{\xi}((u,\infty))\mathrm{d}u.
\end{align*}
Then, use \eqref{lem1a-1} of Lemma \ref{lem1a} with $\alpha=0$ to obtain
\begin{align*}
&  \int_{(1,\infty)}\eta_{\xi}(\mathrm{d}u)=\int_{(1,\infty)}\mathfrak{A}%
(\rho_{\xi})(\mathrm{d}u)\leq C_{1}\int_{(1,\infty)}s^{1/2}\rho_{\xi
}(\mathrm{d}s)\\
&  \qquad=C_{1}\rho_{\xi}((1,\infty))+\frac{C_{1}}{2}\int_{1}^{\infty}%
s^{-1/2}\rho_{\xi}((s,\infty))\mathrm{d}s\leq3cC_{1}.
\end{align*}
Similarly, using \eqref{lem1a-2} of Lemma \ref{lem1a} with $\alpha=1$,
\begin{align*}
&  \int_{(0,1]}u\,\eta_{\xi}(\mathrm{d}u)=\int_{(0,1]}u\,\mathfrak{A}%
(\rho_{\xi})(\mathrm{d}u)\\
&  \qquad\leq C_{2}\left(  \int_{(0,1]}s^{3/2}\rho_{\xi}(\mathrm{d}%
s)+\int_{(1,\infty)}s^{-1/2}\rho_{\xi}(\mathrm{d}s)\right) \\
&  \qquad\leq C_{2}\left(  \frac{3}{2}\int_{0}^{1}s^{1/2}\rho_{\xi
}((s,1])\mathrm{d}s+\int_{(1,\infty)}s^{1/2}\rho_{\xi}(\mathrm{d}s)\right)
\leq6cC_{2}.
\end{align*}
Hence \eqref{Range3-2} is true. It follows that $\mathfrak{M}_{L}%
^{U}(\mathbb{R}^{d})\subset\mathfrak{R}(\mathcal{A}_{1})$.

To see the inclusion is strict, let $\delta_{1}$ be Dirac measure at $1$ and
$\lambda$ a probability measure on $\mathbb{S}$. Consider $\eta\in
\mathfrak{R}(\mathcal{A}_{1})$ defined by
\begin{align*}
\eta(B)  &  =\int_{\mathbb{S}}\lambda(d\xi)\int_{0}^{\infty} 1_{B}%
(r\xi)\mathsf{a}_{1}(\delta_{1})(r)\mathrm{d}r\\
&  =\int_{\mathbb{S}}\lambda(d\xi)\int_{0}^{1} 1_{B}(r\xi)2\pi^{-1}%
(1-r^{2})^{-1/2}\mathrm{d}r.
\end{align*}
Then $\eta\not \in \mathfrak{M}_{L}^{U}(\mathbb{R}^{d})$, since the radial
component has density strictly increasing on $(0,1)$. \qed

\vskip 5mm

\subsection{$\mathcal{A}_{1}$ and $\mathcal{A}_{2}$ as (modified) Upsilon
transformations}


Barndorff-Nielsen, Rosi\'{n}ski and Thorbj\o rnsen \cite{BNRT08} considered
general Upsilon transformations, (see also \cite{BNM08} and \cite{Sa07}).
Given a measure $\tau$ on $(0,\infty)$, a transformation $\Upsilon_{\tau}$
from measures on $\mathbb{R}^{d}$ into $\mathfrak{M}_{L}(\mathbb{R}^{d})$ is
called an Upsilon transformation associated to $\tau$ (or with dilation
measure $\tau$) when
\begin{equation}
\Upsilon_{\tau}(\nu)(B)=\int_{0}^{\infty}\nu(u^{-1}B)\tau(\mathrm{d}u),\qquad
B\in\mathcal{B}(\mathbb{R}^{d}). \label{GeUpsMap}%
\end{equation}
The domain of $\Upsilon_{\tau}$ is the class of $\sigma$-finite measures $\nu$
such that the right-hand side of \eqref{GeUpsMap} is a measure in
$\mathfrak{M}_{L}^{2}(\mathbb{R}^{d})$.

We now see that $\mathcal{A}_{2}$ is an Upsilon transformation and that
$\mathcal{A}_{1}$ is an Upsilon transformation combined with $(1/2)$-transformation.

\begin{thm}
\label{UpsIdent} Let $k=1,2$. Then for $\nu\in\mathfrak{M}_{L}^{k}%
(\mathbb{R}^{d})$
\begin{equation}
\mathcal{A}_{k}(\nu)(B)=\int_{0}^{1} \nu^{(k/2)}(u^{-1}B) 2\pi^{-1}
(1-u^{2})^{-1/2}\mathrm{d}u, \qquad B\in\mathcal{B}(\mathbb{R}^{d}).
\label{PrUpsIdent1}%
\end{equation}

\end{thm}

\noindent\textit{Proof.} Let $(\lambda,\nu_{\xi})$ be a polar decomposition of
$\nu\in\mathfrak{M}_{L}^{k}(\mathbb{R}^{d})$. Then with $c=2\pi^{-1}$
\begin{align*}
\mathcal{A}_{k}(\nu)(B)  &  =c\int_{\mathbb{S}}\lambda(\mathrm{d}\xi) \int%
_{0}^{\infty}1_{B}(r\xi)\mathrm{d}r\int_{(r^{2/k},\infty)}(s^{k}-r^{2}%
)^{-1/2}\nu_{\xi}(\mathrm{d}s)\\
&  =c\int_{\mathbb{S}}\lambda(\mathrm{d}\xi)\int_{0}^{\infty}\nu_{\xi
}(\mathrm{d}s) \int_{0}^{s^{k/2}}1_{B}(r\xi)(s^{k}-r^{2})^{-1/2}\mathrm{d}r\\
&  =c\int_{\mathbb{S}}\lambda(\mathrm{d}\xi)\int_{0}^{\infty}\nu_{\xi
}(\mathrm{d}s) \int_{0}^{1}1_{B}(us^{k/2}\xi)(1-u^{2})^{-1/2}\mathrm{d}u\\
&  =c\int_{0}^{1}(1-u^{2})^{-1/2}\mathrm{d}u\int_{\mathbb{S}}\lambda
(\mathrm{d}\xi)\int_{0}^{\infty}1_{B}(us^{k/2}\xi)\nu_{\xi}(\mathrm{d}s)\\
&  =c\int_{0}^{1}(1-u^{2})^{-1/2}\mathrm{d}u\int_{\mathbb{S}}\lambda
(\mathrm{d}\xi)\int_{0}^{\infty}1_{B}(us\xi){\nu_{\xi}}^{(k/2)}(\mathrm{d}s)\\
&  =c\int_{0}^{1}(1-u^{2})^{-1/2}\mathrm{d}u\int_{\mathbb{R}^{d}}1_{B}%
(ux)\nu^{(k/2)}(\mathrm{d}x),
\end{align*}
which shows \eqref{PrUpsIdent1}. \qed

\begin{cor}
\label{CorUpsA2} The transformation $\mathcal{A}_{2}$ is an Upsilon
transformation with dilation measure $\tau(\mathrm{d}u)$ $=a_{1}
(u;1)\mathrm{d}u.$ In other words, the expression $\widetilde{\nu}
=\mathcal{A}_{2}(\nu)$ for $\nu\in\mathfrak{M}_{L}^{2}(\mathbb{R}^{d})$ is
written as $\widetilde{\nu}(B)=\mathrm{E}\left[  \nu(A^{-1}B)\right]  ,\quad
B\in\mathcal{B}(\mathbb{R}^{d}), $ where $A$ is a random variable with arcsine
density $a_{1}(u;1).$
\end{cor}

\begin{rem}
The mapping $\mathcal{A}_{1}$ is not an Upsilon transformation for any
dilation measure $\tau$. This remarkable result will be proved in Section 3.6,
as a byproduct of Theorem \ref{noncom} shown in Section 3.5.
\end{rem}

\vskip 10mm


\section{Type \textit{A} distributions on $\mathbb{R}^{d}$}

\vskip 5mm

\subsection{Definition and stochastic integral representation via arcsine
transformations}


\begin{defn}
\label{DefTARd} A probability distribution in $I(\mathbb{R}^{d})$ is said to
be a \textit{type A distribution} on $\mathbb{R }^{d}$ if its L\'{e}vy measure
${\nu}$ belongs to $\mathfrak{R}(\mathcal{A}_{1})=\mathfrak{R}(\mathcal{A}%
_{2})$. There is no restriction on ${\Sigma}$ and ${\gamma}$ in its
L\'evy-Khintchine triplet. We denote by $A(\mathbb{R}^{d})$ the class of all
type $A$ distributions on $\mathbb{R}^{d}$.
\end{defn}

In the following, we study a probabilistic interpretation of type $A$
distributions, since they have been defined by an analytic way in terms of
their L\'{e}vy measures above. One probabilistic interpretation is a
representation by stochastic integral with respect to L\'{e}vy processes. The
problem is what the integrand is. We start with this section to answer this question.

Let $T\in(0,\infty)$ and let $f(t)$ be a square integrable function on
$[0,T]$. Then the stochastic integral $\int_{0}^{T}f(t)\mathrm{d}X_{t}^{(\mu
)}$ is defined for any $\mu\in I(\mathbb{R}^{d})$ and is infinitely divisible.
Define the stochastic integral mapping $\Phi_{f}$ based on $f$ as
\[
\Phi_{f}(\mu)=\mathcal{L}\left(  \int_{0}^{T}f(t)\mathrm{d}X_{t}^{(\mu
)}\right)  , \quad\mu\in I(\mathbb{R}^{d}).\label{stochint1}%
\]
If $\mu\in I(\mathbb{R}^{d})$ has the L\'evy-Khintchine triplet $(\Sigma
,\nu,\gamma)$, then $\widetilde{\mu}=\Phi_{f}(\mu)$ has the L\'evy-Khintchine
triplet $(\widetilde{\Sigma},\widetilde{\nu},\widetilde{\gamma})$ expressed
as
\begin{gather}
\widetilde{\Sigma}=\int_{0}^{T}f(t)^{2}\,\Sigma\,\mathrm{d}t,\label{stochint2}%
\\
\widetilde{\nu}(B)=\int_{0}^{T}\mathrm{d}t\int_{\mathbb{R}^{d}}1_{B}
(f(t)x)\,\nu(\mathrm{d}x),\qquad B\in\mathcal{B}(\mathbb{R}^{d}%
),\label{stochint3}\\
\widetilde{\gamma}=\int_{0}^{T}f(t)\mathrm{d}s\left(  \gamma+\int%
_{\mathbb{R}^{d}}x \left(  \frac{1}{1+|f(t)x|^{2}}-\frac{1}{1+|x|^{2}}\right)
\nu(\mathrm{d}x)\right)  . \label{stochint4}%
\end{gather}
(See Proposition 2.17 and Corollary 2.19 of \cite{Sa06a} and Proposition 2.6
of \cite{Sa06b}.)

Let us characterize the class $A(\mathbb{R}^{d})$ as the range of a stochastic
integral mapping.

\begin{thm}
\label{StIntRep} Let
\begin{equation}
\Phi_{\cos}(\mu)=\mathcal{L}\left(  \int_{0}^{1}\cos(2^{-1}\pi t)
\mathrm{d}X_{t}^{(\mu)}\right)  ,\quad\mu\in I(\mathbb{R}^{d}).
\label{StIntRep1}%
\end{equation}
Then $\Phi_{\cos}$ is a one-to-one mapping and
\begin{equation}
A(\mathbb{R}^{d})=\Phi_{\cos}(I(\mathbb{R}^{d})). \label{StIntRep2}%
\end{equation}

\end{thm}

\noindent\textit{Proof.} First let us show that $A(\mathbb{R}^{d})\subset
\Phi_{\cos}(I(\mathbb{R}^{d}))$. Let $\widetilde{\mu}\in A(\mathbb{R}^{d})$
with the L\'evy-Khintchine triplet $(\widetilde{\Sigma},\widetilde{\nu
},\widetilde{\gamma})$. Then $\widetilde{\nu}\in\mathfrak{R}(\mathcal{A}_{2})$
and hence, by Corollary \ref{CorUpsA2},
\[
\widetilde{\nu}(B)=\int_{0}^{1}\nu(u^{-1}B)2\pi^{-1}(1-u^{2})^{-1/2}%
\mathrm{d}u,\quad B\in\mathcal{B}(\mathbb{R}^{d}) \label{StIntRep3}%
\]
with some $\nu\in\mathfrak{M}_{L}^{2}(\mathbb{R}^{d})$. Let $s=g(u)=\int%
_{u}^{1}2\pi^{-1}(1-v^{2})^{-1/2}\mathrm{d}v=$ $2\pi^{-1}\arccos(u)$ for
$0<u<1$. Then $u=\cos(2^{-1}\pi t)$ for $0<t<1$. Thus
\[
\widetilde{\nu}(B)=-\int_{0}^{1}\mathrm{d}g(u)\int_{\mathbb{R}^{d}}
1_{B}(ux)\nu(\mathrm{d}x)=\int_{0}^{1}\mathrm{d}t\int_{\mathbb{R}^{d}}
1_{B}(x\cos(2^{-1}\pi t))\nu(\mathrm{d}x).
\]
That is, \eqref{stochint3} is satisfied with $T=1$ and $f(t)=\cos(2^{-1}\pi
t)$. Using $\nu$, we can find $\Sigma$ and $\gamma$ satisfying
\eqref{stochint2} and \eqref{stochint4}. Let $\mu$ be the distribution in
$I(\mathbb{R}^{d})$ with the L\'evy-Khintchine triplet $(\Sigma,\nu,\gamma)$.
Then $\widetilde{\mu}=\Phi_{\cos}(\mu)$. Hence $A(\mathbb{R}^{d})\subset
\Phi_{\cos}(I(\mathbb{R}^{d}))$.

Conversely, suppose that $\widetilde{\mu}\in\Phi_{\cos}(I(\mathbb{R}^{d}))$.
Then $\widetilde{\mu}=\Phi_{\cos}(\mu)$ for some $\mu\in I(\mathbb{R}^{d})$.
The L\'{e}vy-Khintchine triplets $(\widetilde{\Sigma},\widetilde{\nu
},\widetilde{\gamma})$ and $(\Sigma,\nu,\gamma)$ of $\widetilde{\mu}$ and
$\mu$ are related by \eqref{stochint2}---\eqref{stochint4} with $T=1$ and
$f(s)=\cos(2^{-1}\pi s)$. Then a similar calculus shows that \eqref{StIntRep3}
holds. Hence $\widetilde{\nu}\in\mathfrak{R}(\mathcal{A}_{2})$ and
$\widetilde{\mu}\in A(\mathbb{R}^{d})$, showing that $\Phi_{\cos}%
(I(\mathbb{R}^{d}))\subset A(\mathbb{R}^{d})$.

The mapping $\Phi_{\cos}$ is one-to-one, since $\nu$ is determined by
$\widetilde{\nu}$ (Theorem \ref{t2} with $k=2$) and $\Sigma$ and $\gamma$ are
determined by $\widetilde{\Sigma}$, $\widetilde{\gamma}$, and $\nu$. \qed

\vskip 5mm

\subsection{$\Upsilon^{0}$-transformation}


For later use, we introduce a transformation $\Upsilon^{0}$. Define
\[
\label{Ups0def1}\Upsilon^{0}(\nu)(B)=\int_{0}^{\infty}\nu(u^{-1}%
B)\mathrm{e}^{-u}\mathrm{d}u, \quad B\in\mathcal{B}(\mathbb{R}^{d}).
\]
Let $\mathfrak{M}_{L}^{B}(\mathbb{R }^{d})$ be the class of L\'evy measures of
the Goldie-Steutel-Bondesson class $B(\mathbb{R }^{d})$. In \cite{BNMS06}, it
is shown that $\Upsilon^{0}(\mathfrak{M}_{L}(\mathbb{R }^{d}))=\mathfrak{M}%
_{L}^{B}(\mathbb{R }^{d})$. This is the transformation of L\'evy measures
associated with the stochastic integral mapping $\Upsilon$ from $I(\mathbb{R}%
^{d})$ into $I(\mathbb{R}^{d})$ and it is known that $\Upsilon(I(\mathbb{R}%
^{d}))=B(\mathbb{R}^{d})$ (see \cite{BNMS06}). Both $\Upsilon^{0}$ and
$\Upsilon$ are one-to-one. For $\nu\in\mathfrak{M}_{L}(\mathbb{R}^{d})$ with a
polar decomposition $(\lambda,\nu_{\xi})$, we have the expression
\begin{equation}
\label{Ups0def2}\Upsilon^{0}(\nu)(B)=\int_{\mathbb{S}} \lambda(\mathrm{d}%
\xi)\int_{0}^{\infty} 1_{B}(r\xi)\Upsilon^{0}(\nu_{\xi})(\mathrm{d}r), \quad
B\in\mathcal{B}(\mathbb{R}^{d}),
\end{equation}
where $\Upsilon^{0}$ in the right-hand side acts on $\mathfrak{M}_{L}%
^{2}((0,\infty))$.

\begin{prop}
\label{Ups0} Let $\nu\in\mathfrak{M}_{L}(\mathbb{R}^{d})$. Then $\Upsilon
^{0}(\nu)\in\mathfrak{M}_{L}^{1}(\mathbb{R}^{d})$ if and only if $\nu
\in\mathfrak{M}_{L}^{1}(\mathbb{R}^{d})$.
\end{prop}

\noindent\textit{Proof.} Notice that
\begin{align*}
&  \int_{|x|\leq1} |x|\Upsilon^{0}(\nu)(\mathrm{d}x)=\int_{0}^{\infty}
\mathrm{e}^{-u} \mathrm{d}u\int_{|ux|\leq1} |ux|\nu(\mathrm{d}x)\\
&  \qquad=\int_{0}^{\infty} u\mathrm{e}^{-u} \mathrm{d}u\int_{|x|\leq1/u}
|x|\nu(\mathrm{d}x)=\int_{\mathbb{R}^{d}}|x|\nu(\mathrm{d}x) \int_{0}%
^{1/|x|}u\mathrm{e}^{-u} \mathrm{d}u\\
&  \qquad%
\begin{cases}
\leq\int_{|x|\leq1}|x|\nu(\mathrm{d}x) \int_{0}^{\infty} u\mathrm{e}^{-u}
\mathrm{d}u +\int_{|x|>1}2^{-1}|x|^{-1}\nu(\mathrm{d}x),\\
\geq\int_{|x|\leq1}|x|\nu(\mathrm{d}x)\int_{0}^{1} u\mathrm{e}^{-u}
\mathrm{d}u,
\end{cases}
\end{align*}
to see the equivalence. \qed

\vskip 5mm

\subsection{A representation of completely monotone functions}


In \cite{MS08}, the class of generalized type $G$ distributions on
$\mathbb{R}^{d}$, denoted by $G(\mathbb{R}^{d})$, is defined as follows.
$\mu\in G(\mathbb{R}^{d})$ if and only if the radial component $\nu_{\xi}$ of
the L\'{e}vy measure of $\mu$ satisfies $\nu_{\xi}(\mathrm{d}r)=g_{\xi}%
(r^{2})\mathrm{d}r$, where $g_{\xi}(u)$ is a completely monotone function on
$(0,\infty)$. $\mathfrak{M}_{L}^{G}(\mathbb{R}^{d})$ denotes the class of all
L\'{e}vy measures of $\mu\in G(\mathbb{R}^{d})$. We use the following result
when dealing with $G(\mathbb{R}^{d})$. It is a result on the arcsine
transformation representation of a function $g(r^{2})$ when $g$ is completely
monotone on $(0,\infty)$.

\begin{prop}
\label{Repgr2} Let $g(u)$ be a real-valued measurable function on $(0,\infty
)$. Then the following three conditions are equivalent. \newline\textrm{$(a)$}
The function $g(u)$ is completely monotone on $(0,\infty)$ and satisfies
\begin{equation}
\label{Repgr2-1}\int_{0}^{\infty}(1\land r^{2}) g(r^{2})\mathrm{d}r<\infty.
\end{equation}
\textrm{$(b)$} There exists a completely monotone function $h(s)$ on
$(0,\infty)$ satisfying
\begin{equation}
\label{Repgr2-2}\int_{0}^{\infty}(1\land s) h(s)\mathrm{d}s<\infty
\end{equation}
such that
\[
\label{Repgr2-3}g(r^{2})=\int_{0}^{\infty} a_{1}(r;s)h(s)\mathrm{d}s,\quad
r>0.
\]
\textrm{$(c)$} There exists a measure $\rho$ on $(0,\infty)$ satisfying
\[
\label{Repgr2-4}\int_{0}^{\infty}(1\land s) \rho(\mathrm{d}s)<\infty
\]
such that
\begin{equation}
\label{Repgr2-5}g(r^{2})=\mathsf{a}_{1}(\Upsilon^{0}(\rho))(r),\quad r>0.
\end{equation}

\end{prop}

\noindent\textit{Proof.} (a)\ $\Rightarrow$\ (b): From Bernstein's theorem,
there exists a measure $Q$ on $[0,\infty)$ such that
\begin{equation}
g(u)=\int_{[0,\infty)}\mathrm{e}^{-uv}Q(\mathrm{d}v),\quad u>0.
\label{Repgr2-6}%
\end{equation}
It follows from \eqref{Repgr2-1} that $Q(\{0\})=0$, since $Q(\{0\})=\lim
_{u\rightarrow\infty}g(u)$. We need the fact that the one-dimensional Gaussian
density $\varphi(x;t)$ of mean 0 and variance $t$ is the arcsine transform of
the exponential distribution with mean $t>0$. More precisely,
\begin{equation}
\varphi(x;t)={(2\pi t)}^{-1/2}\mathrm{e}^{-x^{2}/(2t)}={t}^{-1} \int%
_{0}^{\infty}\mathrm{e}^{-s/t}a(x;s)\mathrm{d}s,\;t>0,\;x\in\mathbb{R}.
\label{RGausAT}%
\end{equation}
This is the well-known Box-Muller method to generate normal random variables.
Using \eqref{RGausAT}, we have
\begin{align*}
g(r^{2})  &  =\int_{(0,\infty)}\mathrm{e}^{-r^{2}v}Q(\mathrm{d}v)\\
&  =\int_{(0,\infty)}v^{1/2}Q(\mathrm{d}v)\int_{r^{2}/2}^{\infty}
\mathrm{e}^{-2sv}2\pi^{-1/2}(2s-r^{2})^{-1/2}\mathrm{d}s\\
&  =\int_{(0,\infty)}v^{1/2}Q(\mathrm{d}v)\int_{r^{2}}^{\infty} \mathrm{e}%
^{-sv}\pi^{-1/2}(s-r^{2})^{-1/2}\mathrm{d}s.\\
&  =\int_{r^{2}}^{\infty}\pi^{-1/2}(s-r^{2})^{-1/2}\mathrm{d}s \int%
_{(0,\infty)}\mathrm{e}^{-sv}v^{1/2}Q(\mathrm{d}v)\\
&  =\int_{0}^{\infty}a_{1}(r;s)h(s)\mathrm{d}s,
\end{align*}
where
\begin{equation}
h(s)=2^{-1}\pi^{1/2}\int_{(0,\infty)}\mathrm{e}^{-sv}v^{1/2}Q(\mathrm{d}v).
\label{Repgr2-7}%
\end{equation}
Applying Theorem \ref{DomArcTrn} for $d=1$, we see \eqref{Repgr2-2} from \eqref{Repgr2-1}.

(b)\;$\Rightarrow$\;(c): Since $h(s)$ is completely monotone satisfying
\eqref{Repgr2-2}, there is $\rho\in\mathfrak{M}_{L}^{2}(\mathbb{R})$ such that
$h(s)\mathrm{d}s= \Upsilon^{0}(\rho)$ (see Theorem A of \cite{BNMS06}). Since
$\Upsilon^{0}(\rho)$ is concentrated on $(0,\infty)$, $\rho$ is concentrated
on $(0,\infty)$. Using Proposition \ref{Ups0}, we see that $\int%
_{(0,1]}s\,\rho(\mathrm{d}s)<\infty$.

(c)\ $\Rightarrow$\ (a): It follows from Proposition \ref{Ups0} that
$\int_{(0,1]}s\, \Upsilon^{0}(\rho)(\mathrm{d}s)<\infty$. Hence it follows
from \eqref{Repgr2-5} that $g(r^{2})$ satisfies \eqref{Repgr2-1} (use Theorem
\ref{DomArcTrn} for $d=1$). Finally let us prove that $g(u)$ is completely
monotone. There is a completely monotone function $h(s)$ such that
$\Upsilon^{0}(\rho)(\mathrm{d}s)=h(s)\mathrm{d}s$ (see Theorem A of
\cite{BNMS06} again). Hence from Bernstein's theorem we can find a measure $R
$ on $[0,\infty)$ such that
\[
h(s)=\int_{[0,\infty)}\mathrm{e}^{-sv}R(\mathrm{d}v),\quad s>0.
\]
We have $R(\{0\})=0$ since $\int_{1}^{\infty}h(s)\mathrm{d}s<\infty.$ Thus
\begin{align*}
g(r^{2})  &  =\int_{0}^{\infty}a_{1}(r;s)h(s)\mathrm{d}s =\int_{r^{2}}%
^{\infty}2\pi^{-1}(s-r^{2})^{-1/2}\mathrm{d}s\int_{(0,\infty)} \mathrm{e}%
^{-sv}R(\mathrm{d}v)\\
&  =\int_{(0,\infty)}R(\mathrm{d}v)\int_{r^{2}}^{\infty}2\pi^{-1}
(s-r^{2})^{-1/2}\mathrm{e}^{-sv}\mathrm{d}s\\
&  =\int_{(0,\infty)}\mathrm{e}^{-r^{2}v}2\pi^{-1/2}v^{-1/2}R(\mathrm{d}v),
\end{align*}
where the last equality is from the same calculus as in the proof that
(a)\ $\Rightarrow$\ (b). Now we see that $g(u)$ is completely monotone. \qed

\vskip 5mm

\subsection{A representation of $G(\mathbb{R}^{d})$ in terms of $\mathcal{A}%
_{1}$}


We now give an alternative representation for L\'{e}vy measures of
distributions in $G(\mathbb{R }^{d})$.

\begin{thm}
\label{GenTGRep} Let $\widetilde{\mu}$ be an infinitely divisible distribution
on $\mathbb{R}^{d}$ with the L\'evy-Khintchine triplet $(\widetilde{\Sigma
},\widetilde{\nu},\widetilde{\gamma})$. Then the following three conditions
are equivalent. \newline\textrm{$(a)$} \,\,\,$\widetilde{\mu}\in
G(\mathbb{R}^{d})$. \newline\textrm{$(b)$}\quad$\widetilde{\nu}=\mathcal{A}%
_{1}(\nu)$ with some $\nu\in\mathfrak{M}_{L}^{B}(\mathbb{R}^{d})\cap
\mathfrak{M}_{L}^{1} (\mathbb{R}^{d}) $. \newline\textrm{$(c)$}\quad
$\widetilde{\nu}=\mathcal{A}_{1}(\Upsilon^{0}(\rho))$ with some $\rho
\in\mathfrak{M}_{L}^{1}(\mathbb{R}^{d})$. \newline In condition \textrm{$(b)$}
or \textrm{$(c)$}, the representation of $\widetilde{\nu}$ by $\nu$ or $\rho$
is unique.
\end{thm}

\noindent\textit{Proof.} (a)\;$\Rightarrow$\;(b): By definition of
$G(\mathbb{R}^{d})$, the L\'evy measure $\widetilde{\nu}$ of $\widetilde{\mu}$
has polar decomposition $(\lambda, g_{\xi}(r^{2})\mathrm{d}r)$ where $g_{\xi
}(u)$ is measurable in $(\xi,u)$ and completely monotone in $u>0$. Hence, by
Proposition \ref{Repgr2}, for each $\xi$ we can find a completely monotone
function $\ell_{\xi}(s)$ such that $\int_{0}^{\infty}(1\land s)\ell_{\xi
}(s)\mathrm{d}s <\infty$ and
\[
g_{\xi}(r^{2})=\int_{0}^{\infty} a_{1}(r;s)\ell_{\xi}(s)\mathrm{d}s,\quad
r>0.
\]
The measure $Q_{\xi}$ in the representation \eqref{Repgr2-6} of $g_{\xi}(u)$
has the property that $Q_{\xi}(E)$ is measurable in $\xi$ for every Borel set
$E$ in $[0,\infty)$ (see Remark 3.2 of \cite{BNMS06}). Hence, for any
nonnegative function $f(s,v)$ measurable in $(s,v)$, $\int_{[0,\infty)}
f(s,v)Q_{\xi}(\mathrm{d}v)$ is measurable in $(\xi,s)$. Hence the function
$h_{\xi}(s)$ defined as in \eqref{Repgr2-7} is measurable in $(\xi,s)$. Thus
we have
\[
\widetilde{\nu}(B)=\int_{\mathbb{S}}\lambda(\mathrm{d}\xi)\int_{0}^{\infty}
1_{B}(r\xi) \mathrm{d}r\int_{0}^{\infty} a_{1}(r;s)h_{\xi}(s)\mathrm{d}s.
\]
Now, an argument similar to the proof of Theorem \ref{DomArcTrn} shows that
\[
\int_{\mathbb{S}}\lambda(\mathrm{d}\xi)\int_{0}^{\infty}(1\land s)h_{\xi
}(s)\mathrm{d}s<\infty.
\]
Thus, letting $\nu$ denote the L\'evy measure with polar decomposition
$(\lambda,$ $h_{\xi}(s)\mathrm{d}s)$, we see that $\widetilde{\nu}%
=\mathcal{A}_{1}(\nu)$ and $\nu\in\mathfrak{M}_{L}^{B} (\mathbb{R}^{d}%
)\cap\mathfrak{M}_{L}^{1}(\mathbb{R}^{d})$.

(b)\;$\Rightarrow$\;(c): It follows from $\nu\in\mathfrak{M}_{L}%
^{B}(\mathbb{R}^{d})$ that $\nu=\Upsilon^{0}(\rho)$ for some unique $\rho
\in\mathfrak{M}_{L}^{2}(\mathbb{R}^{d})$ (Theorem A of \cite{BNMS06}). Since
$\nu\in\mathfrak{M}_{L}^{1}(\mathbb{R}^{d})$, we have $\rho\in\mathfrak{M}%
_{L}^{1}(\mathbb{R}^{d})$ from Proposition \ref{Ups0}.

(c)\ $\Rightarrow$\ (a): It follows from $\rho\in\mathfrak{M}_{L}%
^{1}(\mathbb{R}^{d})$ that $\Upsilon^{0}(\rho)\in\mathfrak{M}_{L}%
^{1}(\mathbb{R}^{d})$ (Proposition \ref{Ups0}). Let $(\lambda,\nu_{\xi})$ be
polar decomposition of $\nu=\Upsilon^{0}(\rho)$. Then $\nu_{\xi}
(\mathrm{d}s)=\ell_{\xi}(s)\mathrm{d}s$ where $\ell_{\xi}(s)$ is measurable in
$(\xi,s)$ and completely monotone in $s>0$. Define $g_{\xi}(u)$ by
\[
g_{\xi}(r^{2})=\int_{0}^{\infty}a_{1}(r;s)\ell_{\xi}(s)\mathrm{d}s.
\]
Then $g_{\xi}(u)$ is measurable in $(\xi,u)$. It follows from Proposition
\ref{Repgr2} that $g_{\xi}(u)$ is completely monotone in $u>0$. Hence
$\widetilde{\nu}\in\mathfrak{M}_{L}^{G}(\mathbb{R}^{d})$ and $\widetilde{\mu
}\in G(\mathbb{R}^{d})$. \qed

\vskip 5mm


\subsection{$G(\mathbb{R}^{d})$ as image of $A(\mathbb{R}^{d})$ under a
stochastic integral mapping}


Following \cite{MN08}, we define the transformation $\Upsilon_{\alpha,\beta
}(\nu)$ for $\alpha<2$ and $0<\beta\leq2$. For a measure $\nu$ on
$\mathbb{R}^{d}$ with $\nu(\{0\})=0$ define
\[
\Upsilon_{\alpha,\beta}(\nu)(B)=\int_{0}^{\infty}\nu(s^{-1}B)\beta
s^{-\alpha-1}\mathrm{e}^{-s^{\beta}}\mathrm{d}s,\quad B\in\mathcal{B}
(\mathbb{R}^{d}),\label{UpsMapalbe}%
\]
whenever the right-hand side gives a measure in $\mathfrak{M}_{L}%
(\mathbb{R}^{d})$. This definition is different from that of \cite{MN08} in
the constant factor $\beta$. A special case with $\beta=1$ coincides with the
transformation of L\'{e}vy measures in the stochastic integral mapping
$\Psi_{\alpha}$ studied by Sato \cite{Sa06b}. Of particular interest in this
work is the mapping $\Upsilon_{-2,2}$. Notice that $\Upsilon_{-1,1}%
=\Upsilon^{0}$.

\begin{prop}
\label{dom-2,2} $\Upsilon_{-2,2}(\nu)$ is definable if and only if $\nu
\in\mathfrak{M}_{L}(\mathbb{R}^{d})$. The mapping $\Upsilon_{-2,2}$ is one-to-one.
\end{prop}

\noindent\textit{Proof.} Let $\widetilde{\nu}(B)=\int_{0}^{\infty}\nu
(s^{-1}B)2s\mathrm{e}^{-s^{2}}\mathrm{d}s$. Then
\[
\int_{\mathbb{R}^{d}}f(x)\widetilde{\nu}(\mathrm{d}x)=\int_{0}^{\infty}
2s\mathrm{e}^{-s^{2}}\mathrm{d}s\int_{\mathbb{R}^{d}}f(sx)\nu(\mathrm{d}x)
\]
for all nonnegative measurable functions $f$. Hence
\begin{align*}
&  \int_{\mathbb{R}^{d}}(1\wedge|x|^{2})\widetilde{\nu}(\mathrm{d}x) =\int%
_{0}^{\infty}2s\mathrm{e}^{-s^{2}}\mathrm{d}s\int_{\mathbb{R}^{d}}
(1\wedge|sx|^{2})\nu(\mathrm{d}x)\\
&  \qquad=\int_{0}^{\infty}2s\mathrm{e}^{-s^{2}}\mathrm{d}s\left(
\int_{|x|\leq1/s}|sx|^{2}\nu(\mathrm{d}x)+\int_{|x|>1/s}\nu(\mathrm{d}%
x)\right) \\
&  \qquad=\int_{\mathbb{R}^{d}}|x|^{2}\nu(\mathrm{d}x)\int_{0}^{1/|x|}
2s^{3}\mathrm{e}^{-s^{2}}\mathrm{d}s+\int_{\mathbb{R}^{d}}\nu(\mathrm{d}%
x)\int_{1/|x|}^{\infty} 2s\mathrm{e}^{-s^{2}}\mathrm{d}s.
\end{align*}
Observe that $\int_{0}^{1/|x|}2s^{3}\mathrm{e}^{-s^{2}}\mathrm{d}s$ is
convergent as $|x|\downarrow0$ and $\sim2^{-1}|x|^{-4}$ as $|x|\rightarrow
\infty$ and $\int_{1/|x|}^{\infty}2s\mathrm{e}^{-s^{2}}\mathrm{d}s$ is
$\sim\mathrm{e}^{-1/|x|^{2}}$ as $|x|\downarrow0$ and convergent as
$|x|\rightarrow\infty$. Then we see that $\int_{\mathbb{R}^{d}}(1\wedge
|x|^{2})\widetilde{\nu}(\mathrm{d}x)$ is finite if and only if $\int%
_{\mathbb{R}^{d}}(1\wedge|x|^{2})\nu(\mathrm{d}x)$ is finite. To prove that
$\Upsilon_{-2,2}$ is one-to-one, make a similar argument to the proof of
Proposition 4.1 of \cite{Sa06b}. \qed

\vskip 3mm The following result is needed in showing the characterization of
$G(\mathbb{R }^{d})$ in terms of type $A$ distributions. However, it also
shows that $\mathcal{A}_{1}$ and $\Upsilon^{0}$ are not commutative, while
$\mathcal{A}_{2}$ and $\Upsilon^{0}$ are commutative, both being Upsilon
transformations with domain equal to $\mathfrak{M}_{L}(\mathbb{R}^{d})$.

\begin{thm}
\label{noncom} It holds that
\[
\label{noncom-1}\Upsilon_{-2,2}(\mathcal{A}_{1}(\rho)) =\mathcal{A}%
_{1}(\Upsilon^{0}(\rho))\quad\text{for }\rho\in\mathfrak{M}_{L}^{1}
(\mathbb{R}^{d}).
\]

\end{thm}

\noindent\textit{Proof.} Suppose that $\rho\in\mathfrak{M}_{L}^{1}%
(\mathbb{R}^{d})$ with polar decomposition $(\lambda,\rho_{\xi})$. Let
$\nu=\mathcal{A}_{1}(\rho)$ and $\widetilde{\nu}=\Upsilon_{-2,2}(\nu)$. Then
$\nu$ has polar decomposition $(\lambda,\nu_{\xi})$ with $\nu_{\xi}%
(\mathrm{d}s)=\mathsf{a}_{1}(\rho_{\xi})(s)\mathrm{d}s$. From Theorem 2.6 (ii)
in \cite{MN08}, $\widetilde{\nu}$ has polar decomposition $(\lambda
,\widetilde{\nu}_{\xi})$ given by
\begin{equation}
\widetilde{\nu}_{\xi}(\mathrm{d}r)=rg_{\xi}(r^{2})dr \label{PrThImTA1}%
\end{equation}
with
\begin{equation}
g_{\xi}(r^{2})=2\int_{0}^{\infty}s^{-2}\mathrm{e}^{-r^{2}/s^{2}}\nu_{\xi
}(\mathrm{d}s). \label{PrThImTA2}%
\end{equation}
Using (\ref{PrThImTA1}) and (\ref{PrThImTA2}) we have
\begin{align*}
r{g}_{\xi}(r^{2})  &  =2r\int_{0}^{\infty}\mathrm{e}^{-r^{2}/s^{2}}
s^{-2}\mathsf{a}_{1}(\rho_{\xi})(s)\mathrm{d}s\\
&  =\int_{0}^{\infty}\mathrm{e}^{-t}t^{-1/2}\mathsf{a}_{1}(\rho_{\xi
})(t^{-1/2}r)\mathrm{d}t\\
&  =\int_{0}^{\infty}\mathrm{e}^{-t}t^{-1/2}\mathrm{d}t\int_{0}^{\infty}
a_{1}(t^{-1/2}r;s)\rho_{\xi}(\mathrm{d}s)\\
&  =\int_{0}^{\infty}\mathrm{e}^{-t}\mathrm{d}t\int_{0}^{\infty}
a_{1}(r;ts)\rho_{\xi}(\mathrm{d}s)\\
&  =\int_{0}^{\infty}a_{1}(r;u)\Upsilon^{0}(\rho_{\xi})(\mathrm{d}u),
\end{align*}
since
\[
\int_{0}^{\infty}f(u)\Upsilon^{0}(\rho_{\xi})(\mathrm{d}u) =\int_{0}^{\infty
}\mathrm{e}^{-t}\mathrm{d}t\int_{0}^{\infty}f(ts)\rho_{\xi}(\mathrm{d}s)
\]
for every nonnegative measurable function $f$. It follows that
\[
\widetilde{\nu}(B)=\int_{\mathbb{S}}\lambda(\mathrm{d}\xi)\int_{0}^{\infty}
1_{B}(r\xi)\mathsf{a}_{1}(\Upsilon^{0}(\rho_{\xi}))(\mathrm{d}r),\quad
B\in\mathcal{B}(\mathbb{R}^{d}).
\]
Using \eqref{Ups0def2}, we see that $\widetilde{\nu}=\mathcal{A}_{1}
(\Upsilon^{0}(\rho))$. \qed

\vskip 3mm

The following result shows that $G(\mathbb{R}^{d})$ is the class of
distributions of stochastic integrals with respect L\'{e}vy processes with
type $A$ distribution at time $1.$ This is a multivariate and not necessarily
symmetric generalization of \eqref{Grep}.

\begin{thm}
\label{noncom1} Let
\[
\Psi_{-2,2}(\mu)=\mathcal{L}\left(  \int_{0}^{1}\left(  \log\frac{1}%
{t}\right)  ^{1/2}\mathrm{d}X_{t}^{(\mu)}\right)  ,\quad\mu\in I(\mathbb{R}%
^{d}).
\]
Then $\Psi_{-2,2}$ is one-to-one and
\begin{equation}
G(\mathbb{R}^{d})=\Psi_{-2,2}(A(\mathbb{R}^{d}))=\Psi_{-2,2}(\Phi_{\cos
}(I(\mathbb{R}^{d}))), \label{noncom1-2}%
\end{equation}
where $\Phi_{\cos}$ is defined by \eqref{StIntRep1}. In other words, for any
$\widetilde{\mu}\in G(\mathbb{R}^{d})$ there exists a L\'{e}vy process
$\left\{  X_{t}^{(\mu)}:t\geq0\right\}  $ with type $A$ distribution $\mu$ at
time $1$ such that
\begin{equation}
\widetilde{\mu}=\mathcal{L}\left(  \int_{0}^{1}\left(  \log t^{-1}\right)
^{1/2}\mathrm{d}X_{t}^{(\mu)}\right)  . \label{noncom1-2a}%
\end{equation}

\end{thm}

\noindent\textit{Proof.} Let $g(t)=\int_{t}^{\infty}2u\mathrm{e}^{-u^{2}%
}\mathrm{d}u=\mathrm{e}^{-t^{2}}$. Then the inverse function of $g$ is
$f(t)=\left(  \log t^{-1}\right)  ^{1/2}$ which is square-integrable on
$(0,1)$. Thus, $\Psi_{-2,2}(\mu)$ is definable for all $\mu$. Suppose that
$\widetilde{\mu}\in G(\mathbb{R}^{d})$ with triplet $(\widetilde{\Sigma
},\widetilde{\nu},\widetilde{\gamma})$. Then it follows from Theorems
\ref{GenTGRep} and \ref{noncom} that
\[
\widetilde{\nu}=\mathcal{A}_{1}(\Upsilon^{0}(\rho))=\Upsilon_{-2.2}
(\mathcal{A}_{1}(\rho))
\]
for some $\rho\in\mathfrak{M}_{L}^{1}(\mathbb{R}^{d})$. Let $\nu
=\mathcal{A}_{1}(\rho)$. Since $\widetilde{\nu}=\Upsilon_{-2,2}(\nu)$, we have
\eqref{stochint3} for the function $f(s)=\left(  \log s^{-1}\right)  ^{1/2}$
and $T=1$. Choose $\Sigma$ and $\gamma$ satisfying \eqref{stochint2} and
\eqref{stochint4}. Let $\mu\in I(\mathbb{R}^{d})$ having triplet $(\Sigma
,\nu,\gamma)$. Then $\mu\in A(\mathbb{R}^{d})$ and $\widetilde{\mu}%
=\Psi_{-2,2}(\mu)$. Conversely, we can see that if $\mu\in A(\mathbb{R}^{d})$,
then $\Psi_{-2,2}(\mu)\in G(\mathbb{R}^{d})$. Thus the first equality in
\eqref{noncom1-2} is proved. The second equality follows from
\eqref{StIntRep2} of Theorem \ref{StIntRep}. The one-to-one property of
$\Psi_{-2,2}$ follows from that of $\Upsilon_{-2,2}$ in Proposition
\ref{dom-2,2}. \qed

\vskip 3mm

\begin{rem}
(a) The two representations of $\widetilde{\mu}\in G(\mathbb{R}^{d})$ in
Theorems \ref{GenTGRep} and \ref{noncom1} are related in the following way.
Theorem \ref{noncom1} shows that $\widetilde{\mu}\in G(\mathbb{R}^{d})$ if and
only if $\widetilde{\mu}=\Upsilon_{-2.2}(\Phi_{\cos}(\mu))$ for some $\mu\in$
$I(\mathbb{R}^{d})$. This $\mu$ has L\'{e}vy measure $\rho^{(1/2)}$ if $\rho$
is the L\'{e}vy measure in the representation of $\widetilde{\mu}$ in Theorem
\ref{GenTGRep} (c). For the proof, use Proposition \ref{p-trRep}, Theorems
\ref{StIntRep} and \ref{noncom}.

(b) We have another representation of the class $G(\mathbb{R}^{d}).$ We
introduce the mapping $\mathcal{G}$ as follows. Let $h(t)=\int_{t}^{\infty
}\mathrm{e}^{-u^{2}}\mathrm{d}u,t>0,$ and denote its inverse function by
$h^{\ast}(s).$ For $\mu\in I(\mathbb{R}^{d})$, we define
\[
\mathcal{G}(\mu)=\mathcal{L}\left(  \int_{0}^{\sqrt{\pi}/2}h^{\ast
}(s)\mathrm{d}X_{s}^{(\mu)}\right)  .
\]
It is known that $G(\mathbb{R}^{d})=\mathcal{G}\left(  I(\mathbb{R}%
^{d})\right)  $, see Theorem 2.4 (5) in \cite{MS08}. This suggests us that
$\mathcal{G}$ is decomposed into
\begin{equation}
\mathcal{G=}\Psi_{-2.2}\circ\Phi_{\cos}=\Phi_{\cos}\circ\Psi_{-2.2}
\label{GDecomp}%
\end{equation}
with the same domain $I(\mathbb{R}^{d})$, where $\circ$ means composition of
mappings. This is verified as follows. By Corollary \ref{CorUpsA2},
$\mathcal{A}_{2}$ is an Upsilon transformation and $\mathcal{A}_{2}$
corresponds to $\Phi_{\cos}$ (see (\ref{StIntRep2})). Also, $\Upsilon_{-2.2}$
corresponds to the Upsilon transformation with the dilation measure
$\tau(\mathrm{d}x)=x^{2}\mathrm{e}^{-x^{2}}\mathrm{d}x$. By Proposition 4.1 in
\cite{BNRT08}, we have the second equality in (\ref{GDecomp}).
\end{rem}

\vskip 5mm

\subsection{$\mathcal{A}_{1}$ is not an Upsilon transformation}


By Theorem \ref{noncom}, we obtain the following remarkable result.

\begin{thm}
\label{noncom3} The transformation $\mathcal{A}_{1}$ is not an Upsilon
transformation $\Upsilon_{\tau}$ for any dilation measure $\tau$.
\end{thm}

\noindent\textit{Proof.} Suppose that there is a measure $\tau$ on
$(0,\infty)$ such that
\[
\mathcal{A}_{1}(\rho)(B)=\int_{0}^{\infty}\rho(u^{-1}B)\tau(\mathrm{d}u)
\quad\text{for }B\in\mathfrak{M}_{L}^{1}(\mathbb{R}^{d}).
\]
Then, we can show that
\[
\mathcal{A}_{1}(\Upsilon^{0}(\rho))=\Upsilon^{0}(\mathcal{A}_{1}(\rho))
\quad\text{for }\rho\in\mathfrak{M}_{L}^{1}(\mathbb{R}^{d}).
\]
Indeed, for any nonnegative measurable function $f$
\begin{align*}
\int_{\mathbb{R}^{d}}f(x)\mathcal{A}_{1}(\rho)(\mathrm{d}x)  &  =\int%
_{0}^{\infty}\tau(\mathrm{d}u)\int_{\mathbb{R}^{d}}f(ux)\rho(\mathrm{d}x),\\
\int_{\mathbb{R}^{d}}f(y)\Upsilon^{0}(\rho)(\mathrm{d}y)  &  =\int_{0}%
^{\infty}\mathrm{e}^{-v}\mathrm{d}v\int_{\mathbb{R}^{d}}f(vy)\rho
(\mathrm{d}y),
\end{align*}
and
\[
\mathcal{A}_{1}(\Upsilon^{0}(\rho))(B)=\int_{0}^{\infty}\tau(\mathrm{d}u)
\int_{0}^{\infty}\mathrm{e}^{-v}\mathrm{d}v\int_{\mathbb{R}^{d}}
1_{B}(uvx)\rho(\mathrm{d}x)=\Upsilon^{0}(\mathcal{A}_{1}(\rho))(B).
\]
Then, it follows from Theorem \ref{noncom} that
\[
\Upsilon_{-2,2}(\mathcal{A}_{1}(\rho))=\Upsilon^{0}(\mathcal{A}_{1}
(\rho))\quad\text{for }\rho\in\mathfrak{M}_{L}^{1}(\mathbb{R}^{d}).
\]
Let $\widetilde{\rho}=\mathcal{A}_{1}(\rho).$ If $\int_{\mathbb{R}^{d}%
}\left\vert x\right\vert \rho(\mathrm{d}x)<\infty$, then
\[
\int_{\mathbb{R}^{d}}x\Upsilon^{0}(\widetilde{\rho})(\mathrm{d}x) =\int%
_{0}^{\infty}\mathrm{e}^{-u}\mathrm{d}u\int_{\mathbb{R}^{d}}ux \widetilde{\rho
}(\mathrm{d}x)=\int_{\mathbb{R}^{d}}x\widetilde{\rho}(\mathrm{d}x)
\]
and
\begin{align*}
&  \int_{\mathbb{R}^{d}}x\Upsilon_{-2,2}(\widetilde{\rho}) (\mathrm{d}%
x)=\int_{0}^{\infty}2u\mathrm{e}^{-u^{2}}\mathrm{d}u \int_{\mathbb{R}^{d}%
}ux\widetilde{\rho}(\mathrm{d}x)\\
&  \qquad=\int_{0}^{\infty}2u^{2}\mathrm{e}^{-u^{2}}\mathrm{d}u \int%
_{\mathbb{R}^{d}}x\widetilde{\rho}(\mathrm{d}x)=2^{-1}\pi^{1/2} \int%
_{\mathbb{R}^{d}}x\widetilde{\rho}(\mathrm{d}x).
\end{align*}
Hence $\Upsilon_{-2,2}(\widetilde{\rho})\neq \Upsilon^{0}(\widetilde{\rho})$
whenever $\int_{\mathbb{R}^{d}}x\widetilde{\rho}(\mathrm{d}x)\neq0$ (for
example, choose $\rho=\delta_{e_{1}},e_{1}=(1,0,...,0)$)$.$ This is a
contradiction. Hence the measure $\tau$ does not exist. \qed

\vskip 10mm

\section{Examples}


We conclude this paper with examples for Theorems \ref{GenTGRep} and
\ref{noncom}, where the modified Bessel function $K_{0}$ plays an important
role in the L\'{e}vy measure of infinitely divisible distributions. We only
consider the one-dimensional case of L\'{e}vy measures concentrated on
$(0,\infty).$ Multivariate extensions are possible by using polar decomposition.

By the well-known formula for the modified Bessel functions we have
\[
K_{0}(x)=\frac{1}{2}\int_{0}^{\infty}\mathrm{e}^{-t-x^{2}/(4t)}t^{-1}
\mathrm{d}t,\text{ }x>0.\label{exmp1}%
\]
An alternative expression is%

\begin{equation}
K_{0}(x)=\int_{1}^{\infty}(t^{2}-1)^{-1/2}\mathrm{e}^{-xt} \mathrm{d}t,\text{
}x>0, \label{exmp2}%
\end{equation}
see (3.387.3) in \cite[p.350]{GR07}. It follows that $K_{0}(x)$ is completely
monotone on $(0,\infty)$ and that $\int_{0}^{\infty}K_{0}(x)\mathrm{d}%
x=\pi/2.$

The Laplace transform of $K_{0}$ in $x>0$ is%

\begin{equation}
\label{K0}\varphi_{K_{0}}(s):=\int_{0}^{\infty}\mathrm{e}^{-sx}K_{0}%
(x)\mathrm{d}x=%
\begin{cases}
(1-s^{2})^{-1/2}{\arccos(s)}, & 0<s<1\\
1, & s=1\\
(1-s^{2})^{-1/2}\log(s+(s^{2}-1)^{1/2}), & s>1,
\end{cases}
\end{equation}
see (6.611.9) in \cite[p.695]{GR07}.

\vskip3mm

The following is an example of $\nu$ and $\widetilde{\nu}$ in Theorem
\ref{GenTGRep} (b).

\begin{ex}
\label{ex1} Let
\[
\widetilde{\nu}(\mathrm{d}x)=K_{0}(x)1_{(0,\infty)}(x)\mathrm{d}%
x\label{exmp3a}%
\]
and
\begin{equation}
\nu(\mathrm{d}x)=4^{-1}{\pi}{x^{-1/2}}{\mathrm{e}^{-x^{1/2}}}1_{(0,\infty)}(x)
\mathrm{d}x. \label{exmp4}%
\end{equation}
Then $\nu\in\mathfrak{M}_{L}^{B}(\mathbb{R})\cap\mathfrak{M}_{L}%
^{1}(\mathbb{R})$, and $\widetilde{\nu}=\mathcal{A}_{1}(\nu)\in\mathfrak{M}%
_{L}^{G}(\mathbb{R})$.
\end{ex}

The proof is as follows. Since the function ${x^{-1/2}{\mathrm{e}^{-x^{1/2}}}%
}$ is completely monotone on $(0,\infty)$ and $\int_{0}^{1}x\nu(\mathrm{d}%
x)<\infty$, we have $\nu\in\mathfrak{M}_{L}^{B}(\mathbb{R})\cap\mathfrak{M}%
_{L}^{1}(\mathbb{R})$. Theorem \ref{UpsIdent}, with $k=1,$ gives that for
$B\in\mathcal{B}(\mathbb{R})$
\begin{align*}
\mathcal{A}_{1}(\nu)(B)  &  =\int_{0}^{1}\nu^{(1/2)}(u^{-1}B)2\pi^{-1}
(1-u^{2})^{-1/2}\mathrm{d}u\\
&  =\int_{0}^{1}2\pi^{-1}(1-u^{2})^{-1/2}\mathrm{d}u\int_{0}^{\infty}
1_{u^{-1}B}(s^{1/2})\nu(\mathrm{d}s)\\
&  =\int_{0}^{1}2^{-1}(1-u^{2})^{-1/2}\mathrm{d}u\int_{0}^{\infty}
1_{B}(us^{1/2}){s^{-1/2}}{\mathrm{e}^{-s^{1/2}}}\mathrm{d}s\\
&  =\int_{0}^{1}(1-u^{2})^{-1/2}\mathrm{d}u\int_{0}^{\infty}1_{B}
(r)\mathrm{e}^{-r/u}\mathrm{dr}\\
&  =\int_{0}^{\infty}1_{B}(r)\mathrm{d}r\int_{1}^{\infty}(y^{2}-1)^{-1/2}
\mathrm{e}^{-ry}\mathrm{d}y\\
&  =\int_{0}^{\infty}1_{B}(r)K_{0}(r)\mathrm{d}r=\widetilde{\nu}(B).
\end{align*}
The fact that $\widetilde{\nu}\in\mathfrak{M}_{L}^{G}(\mathbb{R})$ can also be
shown directly, since $K_{0}({x}^{1/2})$ is again completely monotone in
$x\in(0,\infty)$.

It follows from $\widetilde{\nu}\in\mathfrak{M}_{L}^{G}(\mathbb{R})$ that
$\widetilde{\nu}$ is the L\'{e}vy measure of some generalized type $G$
distribution $\widetilde{\mu}$ on $\mathbb{R}$. Using (\ref{K0}), we find that
this $\widetilde{\mu}$ is supported on $[0,\infty)$ if and only if it has
Laplace transform
\begin{equation}
\int_{\lbrack0,\infty)}\mathrm{e}^{-sx}\widetilde{\mu}(\mathrm{d}x)
=\exp\left\{  -\gamma_{0}s+\varphi_{K_{0}}(s)-2^{-1}{\pi}\right\} \nonumber
\end{equation}
for some $\gamma_{0}\geq0.$

\vskip 3mm

\begin{rem}
$\mathcal{A}_{1}(\nu)$ in Example \ref{ex1} actually belongs to a smaller
class $\mathfrak{M}_{L}^{B}(\mathbb{R})$. Therefore, in connection to Theorem
\ref{GenTGRep}, it might be interesting to find a necessary and sufficient
condition on $\nu$ for that $\widetilde{\mu}\in B(\mathbb{R}^{d})$. The $\nu$
in Example \ref{ex1} also belongs to a smaller class than $\mathfrak{M}%
_{L}^{B}(\mathbb{R})\cap\mathfrak{M}_{L}^{1}(\mathbb{R})$. It belongs to the
class of L\'{e}vy measures of distributions in $\mathfrak{R}(\Psi_{-1/2})$
studied in Theorem 4.2 of \cite{Sa06b}.
\end{rem}

\vskip3mm We now give an example of $\rho$ in Theorem \ref{GenTGRep} (c).

\begin{ex}
\label{ex2} Consider the following L\'{e}vy measure in $\mathfrak{M}_{L}%
^{B}(\mathbb{R})$:
\begin{equation}
\rho(\mathrm{d}x)=4^{-1}{{\pi}^{1/2}}{x^{-1/2}}{\mathrm{e}^{-x/4}}
1_{(0,\infty)}(x)\mathrm{d}x. \label{exmp5}%
\end{equation}
Then $\nu$ in (\ref{exmp4}) satisfies $\nu=\Upsilon^{0}(\rho).$
\end{ex}

To prove this, we compute the Upsilon transformation $\Upsilon^{0}$ of $\rho$
as follows:
\begin{align*}
\Upsilon^{0}(\rho)(\mathrm{d}x)  &  =\int_{0}^{\infty}\rho(u^{-1}
\mathrm{d}x)\mathrm{e}^{-u}\mathrm{d}u\\
&  =4^{-1}{{\pi}^{1/2}}{x^{-1/2}}\left(  \int_{0}^{\infty}{u^{-1/2}}
\mathrm{e}^{-u-x/(4u)}\mathrm{d}u\right)  \mathrm{d}x.
\end{align*}
By formula (3.475.15) in \cite[pp 369]{GR07}, we have
\[
\int_{0}^{\infty}{u^{-1/2}}\mathrm{e}^{-u-x/(4u)} \mathrm{d}u={\pi}%
^{1/2}\mathrm{e}^{-x^{1/2}}.
\]
Hence, $\Upsilon^{0}(\rho)(\mathrm{d}x)= 4^{-1}\pi x^{-1/2}{\mathrm{e}%
^{-x^{1/2}}}\mathrm{d}x$ and from (\ref{exmp4}) we have $\nu=\Upsilon^{0}%
(\rho).$

\vskip 3mm

Since $\mathcal{A}_{1}(\nu)=$ $\mathcal{A}_{1}(\Upsilon^{0}(\rho
))=\Upsilon_{-2,2}(\mathcal{A}_{1}(\rho))$ by Theorem \ref{noncom},
$\mathcal{A}_{1}(\rho)$ is also of interest.

\begin{ex}
Let $\rho$ be as in (\ref{exmp5}). Then
\begin{equation}
\mathcal{A}_{1}(\rho)(\mathrm{d}x)={2^{-1}{\pi}^{-1/2}}\mathrm{e}^{-x^{2}%
/8}K_{0}(x^{2}/8)1_{(0,\infty)}(x)\mathrm{d}x.\label{exmp6}%
\end{equation}

\end{ex}

The proof is as follows. We have
\begin{align*}
\mathcal{A}_{1}(\rho)(B)  &  =\int_{0}^{1}\rho^{(1/2)}(u^{-1}B)2\pi
^{-1}(1-u^{2})^{-1/2}\mathrm{d}u\\
&  =\int_{0}^{1}2\pi^{-1}(1-u^{2})^{-1/2}\mathrm{d}u\int_{0}^{\infty}
1_{u^{-1}B}(s^{1/2})\rho(\mathrm{d}s)\\
&  ={2^{-1}{\pi}^{-1/2}}\int_{0}^{1}(1-u^{2})^{-1/2}\mathrm{d}u\int%
_{0}^{\infty} 1_{u^{-1}B}(s^{1/2}){s^{-1/2}}{\mathrm{e}^{-s/4}} \mathrm{d}s\\
&  ={{\pi}^{-1/2}}\int_{0}^{\infty}1_{B}(r)\mathrm{d}r\int_{0}^{1}
u^{-1}(1-u^{2})^{-1/2}\mathrm{e}^{-r^{2}/(4u^{2})}\mathrm{d}u\\
&  ={2^{-1}{\pi}^{-1/2}}\int_{0}^{\infty}1_{B}(r)\mathrm{d}r \int_{1}^{\infty
}y^{-1/2}(y-1)^{-1/2}\mathrm{e}^{-r^{2}y/4}\mathrm{d}y.
\end{align*}
Use (3.383.3) in \cite[pp 347]{GR07} to obtain
\[
\int_{1}^{\infty}y^{-1/2}(y-1)^{-1/2}\mathrm{e}^{-r^{2}y/4}\mathrm{d}%
y=\mathrm{e}^{-r^{2}/8}K_{0}(r^{2}/8).
\]
Thus we obtain (\ref{exmp6}).

\begin{rem}
The $\rho$ in (\ref{exmp5}) also belongs to $\mathfrak{M}_{L}^{B}%
(\mathbb{R})\cap\mathfrak{M}_{L}^{1}(\mathbb{R})$. Therefore $\mathcal{A}%
_{1}(\rho)$\textbf{ }itself is another example of the L\'{e}vy measure of a
generalized type $G$ distribution on $\mathbb{R}$.
\end{rem}



\end{document}